\documentclass[12pt,reqno]{amsart}
\usepackage[headings]{fullpage}
\usepackage{amssymb,amsmath,amscd,bbm,tikz}
\usepackage[all,cmtip]{xy}
\usepackage{url}
\usepackage[bookmarks=true,%
    colorlinks=true,%
    linkcolor=blue,
    citecolor=blue,%
    filecolor=blue,%
    menucolor=blue,%
    urlcolor=blue,%
    breaklinks=true]{hyperref}
\usepackage{slashed}    
\usepackage{verbatim}
\usepackage{mathtools}
\usepackage[normalem]{ulem}   
\graphicspath{{figures/}}

\newtheorem{theorem}{Theorem}[section]
\theoremstyle{definition}
\newtheorem{proposition}[theorem]{Proposition}
\newtheorem{lemma}[theorem]{Lemma}
\newtheorem{definition}[theorem]{Definition}
\newtheorem{remark}[theorem]{Remark}
\newtheorem{corollary}[theorem]{Corollary}

\newtheorem{example}[theorem]{Example}

\def\BN{\mathbbm N}
\def\BZ{\mathbbm Z}
\def\BQ{\mathbbm Q}

\def\BC{\mathbbm C}

\def\calF{\mathcal F}
\def\calL{\mathcal L}

\def\calC{\mathcal C}
\def\calT{\mathcal T}

\def\calS{\mathcal S}

\def\calM{\mathcal M}
\def\calE{\mathcal E}

\def\la{\langle}
\def\ra{\rangle}
\def\ti{\widetilde}

\def\be{\begin{equation}}
\def\ee{\end{equation}}
\def\bea{\begin{equation*}}
\def\eea{\end{equation*}}

\def\vphi{\varphi}
\def\Av{\mathrm{Av}}

\def\rat{\mathrm{rat}}

\def\conn{\mathrm{c}}
\def\loc{\mathrm{loc}}

\begin{document}
\title[Asymptotically multiplicative quantum invariants]{
  Asymptotically multiplicative quantum invariants}
\author{Stavros Garoufalidis}
\address{
International Center for Mathematics, Department of Mathematics \\
Southern University of Science and Technology \\
Shenzhen, China \newline
{\tt \url{http://people.mpim-bonn.mpg.de/stavros}}}
\email{stavros@mpim-bonn.mpg.de}

\author{Seokbeom Yoon}
\address{International Center for Mathematics, Department of Mathematics \\
Southern University of Science and Technology \\
Shenzhen, China \newline
{\tt \url{https://sites.google.com/view/seokbeom}}}
\email{sbyoon15@gmail.com}



\date{November 11, 2022}

\begin{abstract}
  The Euler characteristic and the volume are two best-known multiplicative
  invariants of manifolds under finite covers. On the other hand, quantum invariants
  of 3-manifolds are not multiplicative. We show that a perturbative
  power series, introduced by Dimofte and the first author and shown to be a topological
  invariant of cusped hyperbolic 3-manifolds by Storzer--Wheeler and the
  first author, and conjectured to agree with the asymptotics of the Kashaev invariant
  to all orders in perturbation theory, is 
  asymptotically multiplicative under cyclic covers. Moreover, its coefficients
  are determined
  by polynomials constructed out of twisted Neumann--Zagier data.
  This gives a new $t$-deformation of the perturbative quantum invariants,
  different than the $x$-deformation obtained by deforming the geometric
  representation. We illustrate our results with several hyperbolic knots.
\end{abstract}

\maketitle

{\footnotesize
\tableofcontents
}


\section{Introduction}
\label{sec.intro}

\subsection{Multiplicative invariants and quantum invariants}

A topological invariant $\vphi$ of manifolds is said to be \emph{multiplicative}
if $\vphi({M'})=n\vphi({M})$ for every $n$-sheeted cover $M' \to M$.
Undoubtedly, the best known example of a multiplicative invariant is the Euler
characteristic and, in dimension 3, the volume of a (finite volume) complete
hyperbolic 3-manifold, which is a topological invariant as follows from
Mostow-rigidity~\cite{Thurston}.

On the other hand, quantum invariants of 3-manifolds constructed by a topological
quantum field theory, such as the Witten--Reshetikhin--Turaev
invariant~\cite{Witten:CS,RT:ribbon,Turaev:book}, the
Turaev--Viro invariant~\cite{TV}, and the Kashaev invariant~\cite{K95}, are far
from being multiplicative, even under cyclic covers~\cite{Gilmer}.
However, it turns out that quantum invariants are often
\emph{asymptotically multiplicative}, that is, they satisfy an equation of the form
$\vphi({M'})=n\vphi({M})+O(1)$ for suitable $n$-sheeted covers $M'\to M$, 
reminiscent to invariants of coarse geometry~\cite{Calegari:scl}.

Our goal is to show that some natural perturbative quantum invariants, namely the
ones defined in~\cite{DG1,DG2} are asymptotically multiplicative under cyclic covers
and even more, have a polynomial which determines their values at all cyclic
covers.

To avoid technicalities, we will focus on the 
set $\calM$ of knot complements in rational homology 3-spheres which is closed
 under cyclic coverings of order coprime to a natural number that depends on the
 manifold in question. A 3-manifold $M$ is in $\calM$ if and only if it has betti
number $b_1(M)=1$ and torus boundary; its $n$-fold cyclic cover will be denoted
by $M^{(n)}$.

In \cite{DG1} Dimofte and the first author introduced a power series
\be
\label{Phih}
\Phi_\calT(h)
=\frac{1}{\sqrt{\delta_\calT}}
\Big(1+\sum_{\ell=2}^\infty \Phi_{\calT,\ell} h^{\ell-1}\Big)
=\frac{1}{\sqrt{\delta_\calT}}
\exp\Big(\sum_{\ell=2}^\infty \Phi_{\calT,\ell}^\conn\, h^{\ell-1}\Big)
\ee
associated to an ideal triangulation $\calT$ (and more precisely, to the
Neumann--Zagier data of $\calT$) of a hyperbolic 3-manifold $M \in \calM$.
These power series are constructed by formal Gaussian integration of a multivariate
function which is a product of the infinite Pochhammer symbol (one for each tetrahedron)
times the exponential of a quadratic form. Hence the coefficients $\Phi_{\calT,\ell}$
(resp., $\Phi_{\calT,\ell}^\conn$), the so-called $\ell$-loop invariants
(resp., connected ones) are given by a finite sum over Feynman diagrams of the
contraction of tensors and take values in the invariant trace field $F$ of the
underlying hyperbolic manifold $M$.

It was recently proven in joint work of Storzer--Wheeler and the first author
~\cite{GSW:Phi0} that the power series $\Phi_\calT(h)$
is a topological invariant of cusped hyperbolic 3-manifolds, i.e., that it is
invariant under 2--3 Pachner moves as well as other choices made in its
definition. This series is
conjectured to be the asymptotic expansion to all orders in perturbation theory
of two famous quantum invariants, namely the Kashaev invariant $\la K \ra_N$
of a knot $K$~\cite{K95} and (after complex conjugation)  the Andersen--Kashaev
state-integral~\cite{AK:TQFT}. More precisely,
the Volume Conjecture of Kashaev~\cite{K95} asserts that for a hyperbolic
knot $K$, $\log|\la K \ra_N|$ is asymptotic to $\mathrm{Vol}(K)/(2\pi)$ as $N$ goes
to infinity, and its extension to all orders in $1/N$ asserts that~\cite{Ga:CS,DGLZ}
\be
\la K \ra_N \sim N^\frac{3}{2} e^{\tfrac{\mathrm{V}(K)}{2 \pi i}N}
\Phi_\calT \Big(\frac{2\pi i}{N} \Big) 
\ee
where $\mathrm{V}(K)=i \mathrm{Vol}(K) + \mathrm{CS}(K) \in \BC/(4 \pi^2 \BZ)$ is
the complexified volume of $K$.  
This aspect, together with the rich analytic and arithmetic structure of the
series $\Phi_\calT(h)$ is discussed in detail in~\cite{GZ:kashaev,GZ:qseries}.

\subsection{Asymptotic quantum invariants}
\label{sub.asy}

In our previous paper~\cite{GY21} we studied how the $1$-loop invariant,
$\delta_{\calT}$ in \eqref{Phih}, behaves under finite cyclic covers. To do so, we
showed that the Neumann--Zagier data of the $n$-cyclic cover $\calT^{(n)}$ of an 
ideal triangulation $\calT$ is determined by a twisted version of the Neumann--Zagier
data of $\calT$. Using this, we introduced a twisted version 
$\delta_\calT(t) \in F[t^{\pm 1}]/(\pm t^\BZ)$ of the 1-loop invariant, proved its
topological invariance, and showed that it determines the 1-loop invariant of
$\calT^{(n)}$  for all $n$ by
$\delta_{\calT^{(n)}} = \prod_{\omega^n=1 } \delta_\calT(\omega)$. 
Note that the twisted 1-loop invariant  defined in~\cite{GY21} has a $t-1$ factor
and is $t-1$ times the one used here. Finally, we conjectured that
$\delta_{\calT}(t)$ agrees with the
adjoint twisted Alexander polynomial, a palindromic polynomial  (see
e.g.~\cite{DubYam,Dunfield:twisted}). In the rest of the paper, we will assume that
$\delta_{\calT}(t)$ is palindromic and denote the set of its roots (with possible
repetitions) by $\Lambda=\{\lambda_1^{\pm1},\ldots,\lambda_r^{\pm1}\}$. 
We also denote by $E$ the splitting field of $\delta_\calT(t)$ over the invariant
trace field $F$ of $M$ and by $\|\delta_{\calT}\|$ the maximum of the absolute
values of the roots of $\delta_{\calT}(t)$. The palindromic condition
implies that $\|\delta_{\calT}\| \geq 1$, and when $\delta_{\calT}(t)$ has no roots on
the unit circle, it follows that $\|\delta_{\calT}\| > 1$.

To simplify the statements of our theorems, we will also assume that
(a) $\delta_\calT(t)$ is non-resonant, i.e. that $\prod_{j=1}^r \lambda_j^{n_j} =1$
for integers $n_j$ implies that $n_j=0$ for all $j$, (b) has no roots on the
unit circle, and (c) that the Neumann--Zagier datum is computed with respect to
the longitude. This way our theorems
have clean statements. However, our proofs apply to the resonant case as well as
to Neumann--Zagier data with respect to an arbitrary peripheral curve,
and given explicitly as remarks following the proofs of the theorems.

\begin{theorem}
\label{thm.psi}
For every $\ell \geq 2$ there exists $\Psi^\conn_{\mathcal{T},\ell} \in E$ 
such that
\be
\label{eqn.lloop3}
\Phi^\conn_{\mathcal{T}^{(n)},\ell}  = n \Psi^\conn_{\calT,\ell} 
+ O\big(n^{\ell-1}\|\delta_{\calT}\|^{-n}\big)
\ee
as $n$ tends to infinity, where $\Psi^\conn_{\calT,\ell}$ are
\begin{itemize}
\item
quantum invariants given as a weighted sum over $\ell$-loop
Feynman diagrams of multidimensional integrals of rational differential forms over
tori.
\item
multiplicative under cyclic covers, i.e.,
\be
\label{psin}
\Psi^\conn_{\calT^{(n)},\ell} = n \Psi^\conn_{\calT, \ell} \qquad (n \geq 1) \,.
\ee
\end{itemize}
\end{theorem}

It is likely that the integral representation of $\Psi^\conn_{\calT,\ell}$ can
be evaluated using Grothendieck residues~\cite{Tsikh,Cattani} thus giving an
alternative proof that these invariants take values in the splitting field $E$. 

\subsection{The shape of the quantum invariants of cyclic covers}
\label{sub.rational}

In this section we describe the shape of the $\ell$-loop invariants of $n$-cyclic
covers in terms of the evaluation of polynomials (in a finite dimensional
vector space for each fixed $\ell$) at $1/(1-\lambda_j^n)$ for $j=1,\dots,r$.
We abbreviate the polynomial ring $E[x_1,\ldots,x_r]$ by $E[x]$ and let
$\calF_sE[x]$ denote its subspace spanned by elements of degree at most $s$. 
  
\begin{theorem}
\label{thm.phi}
For every $\ell \geq 2$ there exists a polynomial 
\be
\label{px}
p_{\calT,\ell}(x_1,\dots,x_r,y) \in \calF_{2\ell-2}E[x][y]
\ee
such that for all but finitely many $n$, we have
\be
\label{philn}
\Phi_{ \mathcal{T}^{(n)},\ell} =
p_{\calT,\ell}\left(\frac{1}{1-\lambda_1^n},\dots,\frac{1}{1-\lambda_r^n},n\right) \, .
\ee
The $y$-degree of $p_{\calT,\ell}$ is at most $\ell-1$. 
\end{theorem}

For example, the $2$-loop invariant of  $n$-cyclic covers is given by
\be
\label{2loopn}
\Phi_{\mathcal{T}^{(n)},2} = n \left(
\sum_{1 \leq i \leq j \leq r} \frac{c_{ij}}{(1-\lambda_i^n)(1-\lambda_j^n)}
+ \sum_{1 \leq i \leq r} \frac{c_{i}}{1-\lambda_i^n} + c_{0} \right)
\ee
where $c_{ij}$, $c_i$ and $c_0$ are $(r+1)(r+2)/2$ constants in $E$.

\noindent
Let us make some comments to complement the above theorem.

\noindent
{\bf 1.} The $\ell$-loop invariants are given by a finite sum over the set of
$\ell$-loop Feynman diagrams. The proof of the above
theorem is local, i.e., valid for the contribution of each Feynman diagram, hence
$p_{\calT,\ell}$ is a sum of polynomials that depend on Feynman diagrams. 

\noindent
{\bf 2.}
The coefficients of $p_{\calT,\ell}$ have an integrality property discussed in
 Remark~\ref{rem.integrality} after the proof of the theorem. 

\noindent
{\bf 3.} 
The theorem is stated for each fixed $\ell$ and all but finitely
many $n$. On the other hand, we have good reasons to think that the result holds
for all $n$, and present evidence of this in Section~\ref{sec.examples}.

In the special case when $\delta_{\calT}(t)$ is quadratic (as is the case for all
twist knots), we have an alternative form of the loop invariants of cyclic covers. 

\begin{theorem}
\label{thm.quad}  
If $\delta_{\calT}(t)$ is quadratic, then  for each 
$\ell \geq 2$ there exists a polynomial $q_{\calT,\ell}(x,y) \in \calF_{2\ell-2}F[x][y]$ 
such that  for all but finitely many $n$, we have
\be
\Phi_{\calT^{(n)}, \ell}
= \sum_{t^n=1} q_{\calT,\ell} \left(\frac{1}{\delta_{\calT}(t)},\frac{1}{n} \right) \, .
\ee
\end{theorem}

\subsection{Rationality and determination}
\label{sub.rationality}
	
Theorem~\ref{thm.phi} determines the $\ell$-loop invariants of $n$-cyclic covers
in terms of evaluations of a polynomial $p_{\calT,\ell}$. In this section we address
the opposite problem of determining the polynomial $p_{\calT,\ell}$ from its
evaluations.

A corollary of Theorem~\ref{thm.phi} is that the sequence of $\ell$-loop invariants
of cyclic covers, after multiplied by a suitable power of the $n$-th cyclic resultant 
$N_n(\delta_\calT)=\prod_{\omega^n=1}\delta_\calT(\omega)$
of $\delta_\calT(t)$ has a rational generating series.  
In particular, the sequence of renormalized $\ell$-loop invariants of cyclic covers
is a generalized power sum (in the sense of~\cite{Poorten, Everest} and briefly
reviewed in Section~\ref{sub.psums} below) uniquely determined by a rational function,
which we may think of as a twisted version of the loop invariant.

\begin{proposition}
\label{prop.rational}
\rm{(a)} For every $\ell \geq 2$, there exists a rational function
$\Phi^\rat_{\calT,\ell}(t) \in E(t)$ regular at $t=0$ such that
\be
\label{phirat}
\Phi^\rat_{\calT,\ell}(t)=\sum_{n=0}^\infty 
N_n(\delta_\calT)^{\ell-1} \Phi_{\mathcal{T}^{(n)},\ell} \,t^n \,.
\ee
\rm{(b)} The polynomial $p_{\calT,\ell}(x,y)$ of~\eqref{px} is determined by
$\Lambda$ and $(\ell-1)\binom{r+2\ell-2}{r}$ consecutive values of
$\Phi_{ \mathcal{T}^{(n)},\ell}$.
\end{proposition}

The above proposition gives a $t$-deformation of the perturbative series
$\Phi_\calT(h)$ which is different from the $x$-deformation of $\Phi_\calT(h)$ defined
in~\cite{DG1} and studied in detail~\cite{GGM:peacock}; for instance, see Equations
(123) and (238) for the $4_1$ and the $5_2$ knots, respectively. The $x$-deformed
series $\Phi_\calT(x,h)$ is reciprocal, i.e., satisfies
$\Phi_\calT(x^{-1},h)=\Phi_\calT(x,h)$, as follows from Weyl duality, or from the fact
that $x$ denotes one of the two eigenvalues $x$ and $x^{-1}$ of the holonomy of the
meridian. On the other hand, the rational function of~\eqref{phirat} is not invariant
under $t \mapsto t^{-1}$ as the case of $\ell=2$ and the $5_2$ knot illustrates.

The above corollary determines the polynomial $p_{\calT,\ell}(x,y)$ from 
finitely many values of the $\ell$-loop invariants of cyclic covers, together with
the set $\Lambda$. The next proposition removes the assumption that $\Lambda$ is known,
at the cost of using infinitely many values of the $\ell$-loop invariants of
cyclic covers. This is a generalization of some results of Fried and Hillar who
recover a palindromic polynomial with no cyclotomic factors from its cyclic
resultants~\cite{Fried:cyclic,Hillar}. Its proof uses asymptotics, much in
the spirit of recovering the Poincar\'e map from the asymptotics of the
wave-trace functions~\cite{Guillemin,Ianchenko}. 

\begin{proposition}
\label{prop.determines}
\rm{(a)}
Let $R(x,y) \in \BC(x)[y]$ be a rational function, regular at $x=0$, where
$x=(x_1,\dots,x_r)$ and $\Lambda_+=\{\lambda_1, \ldots,\lambda_r\}$  be a
multiplicatively independent set of nonzero complex numbers with absolute
values less than 1. Then the rational function $R(x,y)$ and the set $\Lambda_+$
are uniquely determined by infinitely many values of
$R(\lambda_1^n,\dots,\lambda_r^n,n)$.
\newline
\rm{(b)} The polynomial  $p_{\calT,\ell}$ is determined by infinitely many values
of $\Phi_{ \mathcal{T}^{(n)},\ell}$ for each fixed $\ell$.
\end{proposition}

Finally, we mention that the structure of Equation~\eqref{px} of perturbative quantum
invariants of cyclic covers is a very general (even if unnoticed) statement of
perturbative quantum field theory.
In particular, one can define asymptotically multiplicative quantum
invariants using perturbation theory of the trivial connection, as described in the 
the power series expansion of the colored Jones polynomial and the
Le--Murakami--Ohtsuki invariant~\cite{LMO}, as well as perturbation theory at abelian
$\mathrm{SU}(2)$-connections, as described by the rational form of the Kontsevich
integral of a knot by Kricker~\cite{Ga-Kricker:rational}. We will discuss this
subject, of a slightly different flavor, in a later investigation. 


\section{A review of the loop invariants}

Let $M$ be an oriented 1-cusped 3-manifold and $\calT$ be an ideal triangulation of
$M$ with $N$ tetrahedra $\Delta_1,\ldots,\Delta_N$ and with $N$ edges $e_1,\ldots, e_N$.
The shape of $\Delta_j$ is described by one complex variable
$z_j \in \BC \setminus \{0,1\}$ and each edge of $\Delta_j$ is assigned with one of the
following parameters with opposite edges having same parameters:
\bea
z_j, \quad z'_j := \frac{1}{1-z_j},  \quad  z''_j := 1- \frac{1}{z_j} \, .
\eea
The shapes $z=(z_1,\ldots,z_N)$ satisfy a system of equations, one per every edge of
$\calT$, hence $N$ equations. Precisely, the equation for an edge $e_i$ is given by
\be \label{eqn.gluing}
\prod_{j=1}^{N} {z_j}^{\mathbf{G}_{ij}}\prod_{j=1}^{N}{z'_j}^{\mathbf{G}'_{ij}}
\prod_{j=1}^{N} {z''_j}^{\mathbf{G}''_{ij}}  = 1 
\ee
where $\mathbf{G}_{ij}$ (resp., $\mathbf{G}'_{ij}, \mathbf{G}''_{ij}$) is the number
of edges of $\Delta_j$ which are incident to $e_i$ in $\calT$ and have shape parameter
$z_j$ (resp., $z'_j, z''_j$). The exponents of~\eqref{eqn.gluing} form three 
$N \times N$ integral matrices $\mathbf{G}, \mathbf{G}'$ and $\mathbf{G}''$ which are
known to be singular. To remove such singularity, we choose a peripheral curve
$\gamma$ of $M$ and replace the last  row of $\mathbf{G}, \mathbf{G}'$ and
$\mathbf{G}''$ (one edge equation) by its completeness equation.
We denote by $\mathbf{G}_\gamma, \mathbf{G}'_\gamma$ and $\mathbf{G}''_\gamma$
the resulting three $N\times N$ integral matrices accordingly. 
Following~\cite{NZ}, we set
\bea
\begin{array}{lll}
\mathbf{A} = \mathbf{G} - \mathbf{G}'\, , &
\mathbf{B} = \mathbf{G}''- \mathbf{G}' \, ,\\[2pt]
\mathbf{A}_\gamma = \mathbf{G}_\gamma - \mathbf{G}'_\gamma \, , &
\mathbf{B}_\gamma = \mathbf{G}''_\gamma- \mathbf{G}'_\gamma 
\end{array}
\eea
and refer to $\mathbf{A}$ and $\mathbf{B}$ (resp., $\mathbf{A}_\gamma$ and
$\mathbf{B}_\gamma$) as Neumann-Zagier matrices of $\calT$ (with respect to $\gamma$).
Finally,  the propagator matrix of $\calT$ with respect to $\gamma$ is defined by
 (\cite{DG1})
\bea
\Pi_{\gamma} = \hbar \left(- \mathbf{B}_\gamma^{-1} \mathbf{A}_\gamma
+ \Delta_{z'} \right)^{-1} 
\eea
where $\hbar$ is a formal variable and  $\Delta_{z'}$ is a diagonal matrix with
diagonal $z'=(z'_1,\ldots, z'_N)$.
  
In what follows, we fix a peripheral curve $\gamma$ of $M$.
Let $G$ be a connected graph with vertex set $V(G)$ and edge set $E(G)$. 
The Feynman rule defines a weight $W_{\calT}(G;\iota)$  for a 
vertex-labeling $\iota : V(G) \rightarrow \{1,\ldots,N\}$ 
\be \label{eqn.feynman2}
W_{\calT}(G;\iota) =\prod_{v\in V(G)}
\Gamma^{(d(v))}_{\iota(v)} \prod_{(v,v') \in E(G)}
\left(\Pi_\gamma \right)_{\iota(v),\, \iota(v')} 
\ee
where $\Gamma^{(d(v))}_{\iota(v)}$  is a rational function in the
shape parameter $z_{\iota(v)}$ depending on the  degree $d(v)$ of $v$ 
(its precise definition will not be needed in this paper) and $(v,v') \in E(G)$ is 
an edge joining $v$ and $v'$. Note that the edge orientation does not matter here, 
as $\Pi_\gamma$ is a symmetric matrix. We refer to \cite[Section 1.6]{DG1} for details. 
Also, a weight $W_{\calT}(G)$ associated to $G$ is defined by 
\be 
\label{eqn.feynman}
W_{\calT}(G) = \frac{1}{ \sigma(G) } \sum _\iota W_{\calT}(G;\iota)
\ee
where the sum is over all vertex-labelings $\iota :V(G) \rightarrow \{1,\ldots, N \}$
and $\sigma(G)$ is the symmetry factor of $G$.

\begin{definition}[\cite{DG1}]
The $\ell$-loop invariant of $\calT$ (with respect to $\gamma$) is defined  by
\be
\label{eqn.lloop}
\Phi_{\calT\!,\ell} =  \mathrm{coeff} \left[
\sum_{G \in \mathcal{G}_\ell} W_{\calT}(G),\,\hbar^{\ell-1} \right]
+ \Gamma^{(0)}
\ee
where $\Gamma^{(0)}$ is the vacuum contribution (see \cite[Eqn 1.18]{DG1}) and 
\bea
\mathcal{G}_\ell =
\{ G: \# \textrm{(1-vertices)} + \# \textrm{(2-vertices)} +  \# \textrm{(loops)}
\leq \ell \} \, .
\eea
Here $\mathrm{coeff}[f(\hbar), \hbar^m]$ denotes  the coefficient of $\hbar^m$ in
a power series $f(\hbar)$.
\end{definition}


\section{Flows and loop invariants of cyclic covers}
\label{sec.flow}

In this section we explain how to express the loop invariants of $n$-cyclic covers
as a sum over Feynmann diagrams of graphs with $\BZ /n\BZ$-flows; see
Theorem~\ref{thm.feynman} below. This is possible since the NZ matrices of cyclic
covers are expressed in terms of the twisted NZ matrices using circulant matrices,
as was found in~\cite{GY21}. We explain this first.

\subsection{Twisted NZ matrices}
  
Let $M$ be an oriented 1-cusped 3-manifold with an ideal triangulation $\calT$
and a surjective morphism $\alpha : \pi_1(M)\rightarrow \BZ$. 
We denote by $M^{(\infty)}$ the infinite cyclic cover of $M$ corresponding to the
kernel of $\alpha$ and $\calT^{(\infty)}$ the ideal triangulation of $M^{\!(\infty)}$
induced from $\calT$. 
Similarly, we denote by $M^{(n)}$ the finite $n$-cyclic cover corresponding to the
subgroup $\alpha^{-1}(n \BZ)$ and $\mathcal{T}^{(n)}$ the ideal triangulation of
$M^{(n)}$ induced from $\calT$.

We choose a fundamental domain of $M$ in $M^{(\infty)}$ and denote by 
$\tilde{e}_i$ (resp., $\tilde{\Delta}_j$) the lift of an edge $e_i$ (resp., a
tetrahedron $\Delta_j$) of $\calT$ to the fundamental domain.  
We also choose a generator $t$ of the deck transformation group $\BZ$ of
$M^{(\infty)}$ so that every tetrahedron of $\calT^{(\infty)}$
is given by $t^k \cdot \tilde{\Delta}_j$ for $k \in \BZ$ and $1 \leq j \leq N$.
We then define an $N\times N$ integral matrix $\mathbf{G}_k$ (resp.,
$\mathbf{G}'_k, \mathbf{G}''_k$) for $k \in \BZ$  by letting its $(i,j)$-entry be the 
number of edges of $t^k \cdot \tilde{\Delta}_j$ which are incident to $\tilde{e}_i$
in $\calT^{(\infty)}$ and have shape parameter $z_j$ (resp., $z'_j, z''_j$). 
Clearly,  $\mathbf{G}_k, \mathbf{G}'_k$ and $\mathbf{G}''_k$ are zero matrices
for all but finitely many $k \in \BZ$.
It follows that (below we view $t$ as a formal variable)
\be
\label{eqn.twistedgluing}
\mathbf{A}(t):=\sum_{k \in \BZ} (\mathbf{G}_k -\mathbf{G}'_k) \, t^k \quad
\textrm{and} \quad   
\mathbf{B}(t):=\sum_{k \in \BZ} (\mathbf{G}''_k -\mathbf{G}'_k) \, t^k 
\ee
are $N \times N$ matrices with entries in $\BZ[t^{\pm1}]$.
We call $\mathbf{A}(t)$ and $\mathbf{B}(t)$ twisted Neumann-Zagier 
matrices of $\calT$. See \cite{GY21} for some basic properties of twisted
Neumann-Zagier matrices.

The twisted Neumann-Zagier matrices $\mathbf{A}(t)$ and $\mathbf{B}(t)$ 
determine the Neumann-Zagier matrices  $\mathbf{A}^{\!(n)}$  and $\mathbf{B}^{(n)}$
of $\calT^{(n)}$ for all $n \geq 1$.
Precisely, for $\mathbf{X} \in \{\mathbf{A},\mathbf{B}\}$ 
\bea
\mathbf{X}^{(n)} = 
\begin{pmatrix*}[c]
\sum_{k \equiv 0} \mathbf{X}_k & \sum_{k \equiv 1}\mathbf{X}_k
& \cdots & \sum_{k \equiv n-1} \mathbf{X}_k \\[1pt]
\sum_{k \equiv n-1} \mathbf{X}_k & \sum_{k \equiv 0} \mathbf{X}_k
& \cdots & \sum_{k \equiv n-2} \mathbf{X}_k\\[1pt]
\vdots & \vdots & \ddots & \vdots \\[1pt]
\sum_{k \equiv 1} \mathbf{X}_k & \sum_{k \equiv 2} \mathbf{X}_k
& \cdots & \sum_{k \equiv 0} \mathbf{X}_k
\end{pmatrix*}, \quad \mathbf{X}(t)=\sum_{k \in \BZ} \mathbf{X}_k \, t^k
\eea
where $\equiv$ means the equality of integers in modulo $n$. In particular,
$\mathbf{X}^{(1)} = \mathbf{X}(1)$ and $\mathbf{X}^{(n)} $ is a block circulant 
matrix for $n \geq 2$. In what follows, we omit the superscript $(n)$ for $n=1$.

\subsection{Block diagonalizations of NZ matrices}

Let $\mu$ and $\lambda$ be peripheral curves of $M$ satisfying $\alpha(\mu)=1$ 
and $\alpha(\lambda)=0$.  We refer to $\mu$ (resp., $\lambda$) as the meridian (resp.,
longitude). Note that $\mu^n$ and $\lambda$ represent peripheral curves of
$M^{(n)}$.  For notational convenience we will confuse $\mu^n$ with  $\mu$; 
for instance,  we simply write $\mathbf{A}^{\!(n)}_{\mu}$ and $\mathbf{B}^{(n)}_{\mu}$
 instead of the Neumann-Zagier matrices  $\mathbf{A}^{\!(n)}_{\mu^n}$ and 
$\mathbf{B}^{(n)}_{\mu^n}$ of $\calT^{(n)}$ with respect to $\mu^n$.

\begin{theorem}
\label{thm.diagonalization}
$(\mathbf{B}^{(n)}_\gamma)^{-1}\mathbf{A}^{\! (n)}_\gamma$ is a block
circulant matrix for $n \geq 2$ and $\gamma \in \{\mu,\lambda\}$.
Moreover, 
\bea
V \, (\mathbf{B}^{(n)}_\gamma)^{-1} \mathbf{A}^{\! (n)}_\gamma \, V^{-1} =
\begin{pmatrix}
\mathbf{B}_\gamma^{-1}\mathbf{A}_\gamma &  & & \\
&  \mathbf{B}(\omega)^{-1} \mathbf{A}( \omega) & & \\
&  &  \ddots & \\
&   &   &  \mathbf{B}(\omega^{n-1})^{-1} \mathbf{A}(\omega^{n-1})
\end{pmatrix} 
\eea
where $\omega = e^\frac{2 \pi \sqrt{-1}}{n}$ and  $V$ is a block Vandermonde
matrix given as in~\eqref{eqn.vandermonde}.
\end{theorem}

As a consequence, the propagator matrix of $\calT^{(n)}$  with respect to
$\gamma \in \{\mu,\lambda\}$
\be
\label{eqn.propagator}
\Pi_\gamma^{(n)} = \hbar \left(- (\mathbf{B}^{(n)}_\gamma)^{-1} 
\mathbf{A}^{(n)}_\gamma + \Delta^{\!(n)}_{z'} \right)^{-1} 
\ee
admits a block diagonalization in terms of 
\bea
\Pi(t) := \hbar \, (- \mathbf{B}(t)^{-1} \mathbf{A}(t)+ \Delta_{z'})^{-1} \, .
\eea
Here $\Delta_{z'}^{\!(n)}$ is the diagonal matrix whose diagonal is $n$ times
repetitions of $z'$.

\begin{corollary}
\label{cor.diagonalization}
We have
\be
\label{eqn.block}
V \, \Pi_\gamma^{(n)} \, V^{-1}= \begin{pmatrix}
\Pi_\gamma &  & & \\
&  \Pi(\omega) & & \\
&  &  \ddots & \\
&   &   &  \Pi(\omega^{n-1})
\end{pmatrix} 
\ee
for $\gamma \in \{\mu,\lambda\}$.
\end{corollary}

\begin{remark}
\label{rmk.lambda} 
The proof of  \cite[Theorem 1.7]{GY21} shows that 
$\mathbf{X}_\lambda = P(t) \mathbf{X}(t)|_{t=1}$ where
\bea
P(t) =  \begin{pmatrix}
1 &  & & \\
& 1 & & \\
&  &  \ddots & \\
\frac{t^{p_1} + t^{q_1}}{t-1}&  \frac{t^{p_2} + t^{q_2}}{t-1} & \cdots
&  \frac{t^{p_N} + t^{q_N}}{t-1}
\end{pmatrix} 
\eea
for some integers $p_i$ and $q_i$.
Hence, $\mathbf{B}_\lambda^{-1} \mathbf{A}_\lambda$ equals to
$\mathbf{B}(t)^{-1} \mathbf{A}(t)$ at $t=1$, and
$\Pi_\lambda = \Pi(1)$.	
This shows that Equation~\eqref{eqn.block} admits full cyclic symmetry for
$\gamma=\lambda$.
\end{remark}

\subsection{Feynman diagrams with flows}
\label{sub.flows}

Let $G$ be a connected graph with vertex set $V(G)$ and edge set $E(G)$.  
We fix an orientation of each edge of $G$ and regard an element of $E(G)$ 
as an oriented edge.
\begin{definition}
A $\BZ  /n\BZ$-\emph{flow} on $G$ is a map $\varphi : E(G) \rightarrow \BZ / n \BZ$
such that  for all $v \in V(G)$
\bea
\sum_{\substack{e \in E(G) \\ e \textrm{ into } v }} \varphi (e)
= \sum_{\substack{e \in E(G) \\ e \textrm{ out of } v }}  \varphi(e) \, .
\eea 
The set of $\BZ \! /n \BZ$-flows on $G$ is naturally an abelian group
isomorphic to $H_1(G; \BZ  / n \BZ) \simeq (\BZ   /n \BZ)^{d}$ where $d$ is
the first betti number of $G$.
\end{definition}

In what follows, we fix a peripheral curve $\gamma \in \{\mu,\lambda\}$.
As a generalizations of~\eqref{eqn.feynman2}, for a vertex-labeling
$\iota :V(G) \rightarrow \{1,\ldots,N\}$, we define a weight 
\be
\label{eqn.newdef1}
W^{(n)}_{\calT} (G;\iota ):=
\frac{1}{n^{d-1}} \sum_{\varphi}  \left( \prod_{v\in V(G)} \Gamma^{(d(v))}_{\iota(v)}
\prod_{(v,v') \in E(G)} \left(\Pi_{\varphi(v,v')}\right)_{\iota(v), \, \iota(v')}
\right)
\ee
where the sum is over all $\BZ \!  /n\BZ$-flows on $G$, $(v,v')$ is an oriented
edge of $G$ running from $v$ to $v'$, and
\bea
\Pi_{\varphi(e)} = 
\left\{
\begin{array}{ll}
\Pi_\gamma & \textrm{if }  \varphi(e)=0 \\[4pt]
\Pi (\omega^{\varphi(e)}) & \textrm{otherwise}  		
\end{array} 
\right. \, .
\eea
Also, as in~\eqref{eqn.feynman}, we set
\be \label{eqn.newdef2}
W^{(n)}_{\calT}(G) := \frac{1}{\sigma(G)}
\sum_{\iota} W_{\calT}^{(n)}(G;\iota)
\ee
where the sum is over all vertex-labelings $\iota : V(G)\rightarrow \{1,\ldots,N\}$.
We remark that for $n=1$ there is only one flow, the trivial one (assigning 0 to
all edges), hence \eqref{eqn.newdef1} and~\eqref{eqn.newdef2} reduce
to~\eqref{eqn.feynman2} and \eqref{eqn.feynman}, respectively.

\begin{lemma}
\label{lem.orientation}
$W^{(n)}_{\calT}(G)$ does not depend on the choice of edge
orientation of $G$. 
\end{lemma}

\begin{example}
Let us consider a Feynman diagram $G$ with edge orientation as in
Figure~\ref{fig.diagram}. 
For notational simplicity we let a vertex-labeling $\iota$  assign the vertices
to $(i,j) \in \{1,\ldots,N\}^2$ and  a map $\varphi$ assign the oriented edges
to $(a,b,c) \in (\BZ/n \BZ)^3$ as in Figure~\ref{fig.diagram}. 

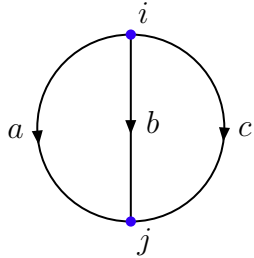
\begin{figure}[htpb!]
\begingroup%
  \makeatletter%
  \providecommand\color[2][]{%
    \errmessage{(Inkscape) Color is used for the text in Inkscape, but the package 'color.sty' is not loaded}%
    \renewcommand\color[2][]{}%
  }%
  \providecommand\transparent[1]{%
    \errmessage{(Inkscape) Transparency is used (non-zero) for the text in Inkscape, but the package 'transparent.sty' is not loaded}%
    \renewcommand\transparent[1]{}%
  }%
  \providecommand\rotatebox[2]{#2}%
  \newcommand*\fsize{\dimexpr\f@size pt\relax}%
  \newcommand*\lineheight[1]{\fontsize{\fsize}{#1\fsize}\selectfont}%
  \ifx\svgwidth\undefined%
    \setlength{\unitlength}{96.47867879bp}%
    \ifx\svgscale\undefined%
      \relax%
    \else%
      \setlength{\unitlength}{\unitlength * \real{\svgscale}}%
    \fi%
  \else%
    \setlength{\unitlength}{\svgwidth}%
  \fi%
  \global\let\svgwidth\undefined%
  \global\let\svgscale\undefined%
  \makeatother%
  \begin{picture}(1,0.81025763)%
    \lineheight{1}%
    \setlength\tabcolsep{0pt}%
    \put(0,0){\includegraphics[width=\unitlength,page=1]{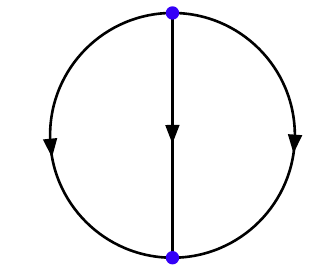}}%
    \put(0.54310392,0.81763728){\makebox(0,0)[lt]{\lineheight{1.25}\smash{\begin{tabular}[t]{l}$i$\end{tabular}}}}%
    \put(0.54050546,-0.06605175){\makebox(0,0)[lt]{\lineheight{1.25}\smash{\begin{tabular}[t]{l}$j$\end{tabular}}}}%
    \put(0.03312405,0.36834955){\makebox(0,0)[lt]{\lineheight{1.25}\smash{\begin{tabular}[t]{l}$a$\end{tabular}}}}%
    \put(0.93464229,0.37467556){\makebox(0,0)[lt]{\lineheight{1.25}\smash{\begin{tabular}[t]{l}$c$\end{tabular}}}}%
    \put(0.57429126,0.38584028){\makebox(0,0)[lt]{\lineheight{1.25}\smash{\begin{tabular}[t]{l}$b$\end{tabular}}}}%
  \end{picture}%
\endgroup%

\caption{A 2-loop Feynman diagram.}
\label{fig.diagram}
\end{figure}	

Since $\varphi$ is a flow if and only if $a+b+c \equiv 0$ in modulo $n$, we have
\begin{equation*}
W^{(n)}_{\calT}(G)  = \frac{1}{8n} \sum_{1 \leq i, j \leq N} \,
\sum_{a,b \,  \in \BZ \!  /n\BZ}
\Gamma_{i}^{(3)}\Gamma_{j}^{(3)} (\Pi_a)_{ij} (\Pi_b)_{ij} (\Pi_{-a-b})_{ij}\, .
\end{equation*}
Note that $\sigma(G)=8$. In particular, if the peripheral curve $\gamma$ is
chosen to be the longitude, we have (see Remark~\ref{rmk.lambda})
\begin{equation}
\label{eqn.example}
W^{(n)}_{\calT}(G) = \frac{1}{ 8n} \sum_{1 \leq i, j \leq N} \,
\sum_{a,b \,  \in \BZ \! /n\BZ} \Gamma_{i}^{(3)} \Gamma_{j}^{(3)} \Pi(\omega^a)_{ij}
\Pi(\omega^b)_{ij} \Pi(\omega^{-a-b})_{ij}  \, .
\end{equation}
\end{example}

\begin{theorem}
\label{thm.feynman}
The $\ell$-loop invariant of $\mathcal{T}^{(n)}$ satisfies
\be
\label{eqn.lloop2}
\Phi_{\calT^{(n)}\!, \, \ell}   =  \mathrm{coeff}
\left[ \sum_{G \in \mathcal{G}_\ell} W^{(n)}_{\calT}(G), \,\hbar^{\ell-1}
\right]+ \Gamma^{(0)}
\ee
for all $n \geq 1$.
\end{theorem}


\section{Proofs of the flow statements}

We devote this section to prove theorems in Section~\ref{sec.flow}.

\subsection{Some facts on block circulant matrices}

We here list some elementary facts on block circulant matrices.  We refer
to~\cite{Davis} for details. Let $C$ be a block circulant matrix of the form
\bea
C = \begin{pmatrix}
C_0 & C_1 & \cdots & C_{n-1} \\
C_{n-1} & C_0 & \cdots & C_{n-2} \\
\vdots & \vdots & \ddots & \vdots\\
C_1 & C_2 & \cdots & C_{0} 
\end{pmatrix} \,.
\eea
It is known that $C$ is block diagonalizable by conjugating a block Vandermonde
matrix
\begin{equation}
\label{eqn.vandermonde}
V = \begin{pmatrix}
I & I & \cdots & I \\
I & \omega I  & \cdots & \omega^{n-1}  I\\
\vdots & \vdots & \ddots & \vdots\\
I & \omega^{n-1}I & \cdots & \omega^{(n-1)(n-1)} I
\end{pmatrix} , \quad \omega = e^\frac{2 \pi \sqrt{-1}}{n} \,.
\end{equation}
Explicitly, we have
\bea
V C V^{-1} =  \begin{pmatrix}
r(\omega^0) &  & & \\
&  r(\omega^1)& & \\
&  &  \ddots & \\
&   &   & r(\omega^{n-1})
\end{pmatrix}
\eea
where $r(t) = C_0 + C_1 t + \cdots+ C_{n-1}t^{n-1}$ is the \emph{representer} of $C$. 
It follows that  one can recover the matrix $C$ from its representer $r(t)$ by
\begin{equation}
\label{eqn.recover}
C_i  = \frac{1}{n} \left( r(\omega^0) + \omega^i r(\omega^1) + \cdots
  + \omega^{i(n-1)} r(\omega^{n-1}) \right) 
\end{equation}
for $0 \leq i < n$.
Note that if $r(t)$ is constant, i.e. $C_1 = \cdots =C_{n-1} =0$, then $C$
commutes with $V$.

\subsection{Circulant structure of NZ matrices}

In this section we prove Theorem~\ref{thm.diagonalization} and its
Corollary~\ref{cor.diagonalization}. 

To begin the proof of Theorem~\ref{thm.diagonalization}, 
we first consider the case of $\gamma = \mu$.
Let $a_i$ (resp., $b_i$) be the $i$-the row of $\mathbf{A}^{\! (n)}$ (resp.,
$\mathbf{B}^{(n)}$) and $a_{\mu^n}$ (resp., $b_{\mu^n}$) be the last row of
$\mathbf{A}^{\!(n)}_{\mu}$ (resp., $\mathbf{B}^{(n)}_{\mu}$).
Recall that  $a_{\mu^n}$ and $b_{\mu^n}$ represent the completeness
equation of $\mu^n$, hence they are $n$-times repetitions of $a_{\mu}(=a_{\mu^1})$
and $b_{\mu}(=b_{\mu^1})$ respectively:
\bea
a_{\mu^n} = ( \underbrace{a_\mu \ \cdots \ a_\mu}_{n}) , \quad
b_{\mu^n} = ( \underbrace{b_\mu \ \cdots \ b_\mu}_{n}) \, .
\eea
For a vector $v$ we denote by $\mathbf{O}[v]$ a square matrix whose last row is $v$
and the other rows are trivial.
From the symplectic property of Neumann-Zagier matrices \cite[Theorem 2.2]{NZ}, we
have $a_i b_j^T = b_i a_j^T$ and $a_{\mu^n} \, b_j^T = b_{\mu^n} \, a_j^T$. 
It follows that
\bea
(\mathbf{A}^{\! (n)} + \mathbf{O}[{a_{\mu^n}}])\, (\mathbf{B}_\mu^{(n)})^T
=(\mathbf{B}^{(n)} + \mathbf{O}[b_{\mu^n}]) \, (\mathbf{A}_\mu^{\! (n)})^T 
\eea
and
\be
\label{eqn.a1}
(\mathbf{B}_\mu^{(n)})^{-1}  \mathbf{A}_\mu^{\!(n)}=
\left((\mathbf{B}_\mu^{(n)})^{-1}  \mathbf{A}_\mu^{\!(n)}\right)^T =(\mathbf{B}^{(n)}
+ \mathbf{O}[b_{\mu^n}])^{-1}(\mathbf{A}^{\!(n)} +  \mathbf{O}[a_{\mu^n}]).
\ee
On the other hand, a simple matrix computation shows that 
\bea
V \mathbf{O}[a_{\mu^n}] V^{-1} = \begin{pmatrix}
\mathbf{O}[a_{\mu}] & 0 & \cdots & 0  \\
\omega^{-1}\mathbf{O}[a_{\mu}] & 0 & \cdots & 0  \\
\vdots & \vdots & \ddots & \vdots\\
\omega^{-(n-1)} \mathbf{O}[a_{\mu}] & 0 & \cdots & 0  
\end{pmatrix} 
\eea
and thus
\begin{align*}
V (\mathbf{A}^{\!(n)} + \mathbf{O}[a_{\mu^n}]) V^{-1} = 
\begin{pmatrix}
\mathbf{A}(1) + \mathbf{O}[a_{\mu}]&  0  & \cdots& 0\\
\omega^{-1} \mathbf{O}[a_{\mu}] &  \mathbf{A}(\omega^1)& & \\
\vdots &  &  \ddots & \\
\omega^{-(n-1)} \mathbf{O}[a_{\mu}] &   &   & \mathbf{A}(\omega^{n-1})  
\end{pmatrix} \, . 
\end{align*} 
Similarly, we have
\begin{align*}
V (\mathbf{B}^{(n)} + \mathbf{O}[b_{\mu^n}]) V^{-1} &= 
\begin{pmatrix}
\mathbf{B}(1) + \mathbf{O}[b_{\mu}]&  0  & \cdots& 0\\
\omega^{-1} \mathbf{O}[b_{\mu}] &  \mathbf{B}(\omega^1)& & \\
\vdots &  &  \ddots & \\
\omega^{-(n-1)} \mathbf{O}[b_{\mu}] &   &   & \mathbf{B}(\omega^{n-1})  
\end{pmatrix} \, . 
\end{align*}
It follows that 
\begin{align*}
V (\mathbf{B}^{(n)} + \mathbf{O}[b_{\mu^n}])^{-1} V^{-1} &= 
\begin{pmatrix}
(\mathbf{B}(1) +  \mathbf{O}[b_{\mu}])^{-1} &  0  & \cdots& 0\\
C_1 &  \mathbf{B}(\omega^1)^{-1}& & \\
\vdots &  &  \ddots & \\
C_{n-1} &   &   & \mathbf{B}(\omega^{n-1}) ^{-1}
\end{pmatrix} 
\end{align*}
where $C_k$ ($1 \leq k <n$) is a matrix satisfying
\be
\label{eqn.a2}
C_k \, ( \mathbf{B}(1)+ \mathbf{O}[b_\mu])
+\omega^{-k} \mathbf{B}(\omega^{k})^{-1} \mathbf{O}[b_\mu]= 0. 
\ee
Combining the above, we obtain 
\begin{align*}
V (\mathbf{B}^{(n)}_{\mu})^{-1} \mathbf{A}^{\!(n)}_{\mu}V^{-1} 
&= V(\mathbf{B}^{(n)} + \mathbf{O}[b_{\mu^n}])^{-1} (\mathbf{A}^{\!(n)}
+ \mathbf{O}[a_{\mu^n}])V^{-1} \\
&= 
\begin{pmatrix}
D_0 & 0  &  \cdots & 0 \\
D_1 &  \mathbf{B}(\omega^1)^{-1} \mathbf{A}(\omega^1)& & \\
\vdots &  &  \ddots & \\
D_{n-1} &   &   & \mathbf{B}(\omega^{n-1}) ^{-1} \mathbf{A}(\omega^{n-1})
\end{pmatrix}
\end{align*}
where 
\begin{align*}
D_0&=(\mathbf{B}(1)+\mathbf{O}[b_\mu])^{-1} ( \mathbf{A}(1)+\mathbf{O}[a_\mu]) , \\
D_k &=  C_k (\mathbf{A}(1) + \mathbf{O}[a_\mu]) + \omega^{-k}\mathbf{B}(\omega^k)^{-1}
\mathbf{O}[a_\mu], \quad  k=1,\ldots,n-1.
\end{align*}
On the other hand, from Equation~\eqref{eqn.a1}  we have
\bea
D_0=(\mathbf{B}(1)+\mathbf{O}[b_\mu])^{-1} ( \mathbf{A}(1)+\mathbf{O}[a_\mu])
= \mathbf{B}_\mu^{-1} \mathbf{A}_\mu
\eea
(recall that $\mathbf{A}^{(1)} = \mathbf{A}(1)$ and $\mathbf{B}^{(1)} = \mathbf{B}(1)$) 
and  from Equation~\eqref{eqn.a2} we have 
\begin{align*}
D_k &=  C_k (\mathbf{A}(1) + \mathbf{O}[a_\mu])
+ \omega^{-k}\mathbf{B}(\omega^k)^{-1} \mathbf{O}[a_\mu] \\
& = \omega^{-k} \mathbf{B}(\omega^k)^{-1} \left( -\mathbf{O}[b_\mu]
(\mathbf{B}(1) + \mathbf{O}[b_\mu])^{-1} (\mathbf{A}(1) + \mathbf{O}[a_\mu])
+ \mathbf{O}[a_\mu] \right)\\
& =  \omega^{-k} \mathbf{B}(\omega^k)^{-1} \left( -\mathbf{O}[b_\mu]
\mathbf{B}_\mu^{-1} \mathbf{A}_\mu + \mathbf{O}[a_\mu] \right) \\
&=0 
\end{align*}
where the last equation follows from the last row of an obvious identity
$\mathbf{B}_\mu \mathbf{B}_\mu^{-1} \mathbf{A}_\mu = \mathbf{A}_\mu$, i.e.,
$b_\mu \mathbf{B}_\mu^{-1} \mathbf{A}_\mu = a_\mu$.
This proves Theorem~\ref{thm.diagonalization} for the case  of $\gamma =\mu$.

We prove the case of $\gamma =\lambda$ similarly. Let $a_\lambda$ and
$b_\lambda$ be the last row of $\mathbf{A}_\lambda^{\!(n)}$ and
$\mathbf{B}_\lambda^{(n)}$, respectively.
We may write $a_\lambda$ and $b_\lambda$ as 
\bea
a_{\lambda} = ( v_1 \ \cdots \  v_n) \textrm{ and }
b_{\lambda} = ( w_1 \ \cdots \  w_n)
\eea
where $v_i$ and $w_i$ are vectors of length $N$. Note that 
$\mathbf{v} := \sum_{i=1}^n v_i$ and $\mathbf{w} := \sum_{i=1}^n w_i$ represent the
complete equation of $\lambda$ in $\mathcal{T}$ and thus
\bea
a_{\lambda^n} := (  \underbrace{\mathbf{v} \ \cdots \ \mathbf{v} }_{n})\textrm{ and }
b_{\lambda^n} := (  \underbrace{\mathbf{w} \ \cdots \ \mathbf{w} }_{n}) 
\eea
represent $n$-parallel copies of $\lambda$ in $\mathcal{T}^{(n)}$.
It follows that
\bea
(\mathbf{A}^{\!(n)} + \mathbf{O}[{a_{\lambda^n}}])\, (\mathbf{B}_\lambda^{(n)})^T
=(\mathbf{B}^{(n)} + \mathbf{O}[b_{\lambda^n}]) \, (\mathbf{A}_\lambda^{\!(n)})^T 
\eea
and
\bea 
(\mathbf{B}_\lambda^{(n)})^{-1}  \mathbf{A}_\lambda^{\!(n)}=
\left((\mathbf{B}_\lambda^{(n)})^{-1}  \mathbf{A}_\lambda^{\!(n)}\right)^T =
(\mathbf{B}^{(n)} + \mathbf{O}[b_{\lambda^n}])^{-1}(\mathbf{A}^{\!(n)}
+  \mathbf{O}[a_{\lambda^n}]).
\eea
Then the rest computation is  exactly same as the case of $\gamma = \mu$.
\qed

We now turn to Corollary~\ref{cor.diagonalization}.
The block Vandermonde matrix $V$~given as in~\eqref{eqn.vandermonde}
commutes with the diagonal matrix $\Delta^{(n)}_{z'}$.
Thus Theorem~\ref{thm.diagonalization} implies
\bea
V \left((\mathbf{B}^{(n)}_\gamma)^{-1}
  \mathbf{A}^{\!(n)}_\gamma + \Delta^{\!(n)}_{z'}\right)
V^{-1}=\begin{pmatrix}
\mathbf{B}_\gamma^{-1} \mathbf{A}_\gamma+ \Delta_{z'}&  &  \\
&  \ddots & \\
&   &  \mathbf{B}(\omega^{n-1})^{-1} \mathbf{A}(\omega^{n-1}) + \Delta_{z'}
\end{pmatrix} 
\eea
and  Corollary~\ref{cor.diagonalization} follows.
\qed

\subsection{From flows to loop invariants of cyclic covers}

In this section we prove Lemma~\ref{lem.orientation} and Theorem~\ref{thm.feynman}. 

Let $e=(v,v')$ be an oriented edge of $G$ and  $-e=(v',v)$ be the same edge with
reversed orientation. It follows from \cite[Theorem~1.2]{GY21} that
\begin{equation*}
\mathbf{B}(1/t)^{-1}\mathbf{A}(1/t)=(\mathbf{B}(t)^{-1}\mathbf{A}(t))^T
\end{equation*}
and hence $\Pi(1/t) = \Pi(t)^T$.
In particular, for any $\BZ /n \BZ$-flows $\varphi$ on $G$ we have
\begin{equation*}
\Pi(\omega^{\varphi(e)})_{ \iota(v), \iota(v')}
=\Pi(\omega^{-\varphi(e)})_{ \iota(v') \iota(v)}
=\Pi(\omega^{\varphi(-e)})_{ \iota(v'), \iota(v)}\, .
\end{equation*}
Combining the above with the fact that $\Pi_\gamma$ is symmetric, we obtain
Lemma~\ref{lem.orientation}.
\qed

We now turn to Theorem~\ref{thm.feynman}.
Corollary~\ref{cor.diagonalization} shows that  $\Pi_\gamma^{(n)}$ 
is a block circulant matrix whose first row is 
\bea
\frac{1}{n} \begin{pmatrix}
\displaystyle\sum_{k=0}^{n-1} \Pi_k  &
\displaystyle\sum_{k=0}^{n-1} \omega^k \Pi_k &  \cdots  &
\displaystyle\sum_{k=0}^{n-1} \omega^{(n-1)k}\Pi_k 
\end{pmatrix}
\eea
where $\Pi_0 = \Pi_\gamma$ and $\Pi_k = \Pi(\omega^k)$ for $1 \leq k \leq n-1$. 
In addition, its $(aN+i, bN+j)$-entry, where $1 \leq i,j \leq N$
 and $0 \leq a ,b \leq n-1$, is 
\bea
\left( \Pi_\gamma^{(n)} \right)_{aN+i,\, bN+j} = \frac{1}{n}
\left(\sum_{k=0}^{n-1} \omega^{(b-a)k}\Pi_k \right)_{i,j} \, .
\eea
For a Feynman diagram $G$ with vertex set $V=V(G)$ and edge set $E=E(G)$ 
we have
\begin{align}
W_{\calT^{(n)}}(G) & = \frac{1}{ \sigma(G)} \sum_{\iota \in [\![1,nN ]\!]^V }
W_{\calT^{(n)}}(G;\iota) \nonumber \\
&=\frac{1}{\sigma(G)} \sum_{\iota \in [\![1,nN ]\!]^V} \left(\displaystyle
\prod_{v \in V} \Gamma_{\iota(v)}^{(d(v))}
\displaystyle\prod_{(v,v') \in E} \left(\Pi_\gamma^{(n)}\right)_{\iota(v),\,\iota(v')}
\right) \, . \label{eqn.proc}
\end{align}
where $[\![i,j]\!]^V$ denotes the set of maps from the vertex set $V$ to
$\{i,\ldots,j\}$. Since $\Pi_\gamma^{(n)}$ is a symmetric matrix, we may regard
$(v,v')\in E$ as an oriented edge. Writing $\iota(v) = p(v)N +\underline{\iota}(v)$
for $ 1 \leq \underline{\iota}(v) \leq N$, we can rewrite Equation~\eqref{eqn.proc}
as
\begin{align*}
& \frac{1}{ \sigma(G)} \sum_{\underline{\iota} \in [\![ 1,N ]\!]^V}
\sum_{p \in [\![0,n-1]\!]^V} \left( \displaystyle\prod_{v \in V}
\Gamma_{p(v)N+\underline{\iota}(v)}^{(d(v))} 
\displaystyle\prod_{(v,v')\in E}
\left(\Pi_\gamma^{(n)}\right)_{p(v)N+\underline{\iota}(v),\, p(v')N+\underline{\iota}(v')}
\right) \\
&=\frac{1}{\sigma(G)} \sum_{\underline{\iota} \in  [ \![ 1,N ]\!]^V}
\displaystyle\prod_{v \in V} \Gamma_{\underline{\iota}(v)}^{(d(v))}
\left( \sum_{p \in [\![0,n-1]\!]^V}
\displaystyle\prod_{(v,v')\in E}
\left(\Pi_\gamma^{(n)}\right)_{p(v)N+\underline{\iota}(v),\, p(v')N+\underline{\iota}(v')}
\right)\\
&=\frac{1}{\sigma(G)} \sum_{\underline{\iota} \in  [ \![ 1,N ]\!]^V}
\displaystyle\prod_{v \in V} \Gamma_{\underline{\iota}(v)}^{(d(v))}
\left(\sum_{p \in [\![0,n-1]\!]^V} 	\displaystyle\prod_{(v,v') \in E} 
\left( \frac{1}{n}\sum_{k=0}^{n-1}
\omega^{(p(v')-p(v))k}\Pi_k \right)_{\underline{\iota}(v) ,\, \underline{\iota}(v')}
\right) \,.
\end{align*}
On the other hand, we have
\begin{align*}
& \sum_{p \in [\![0,n-1]\!]^V} 	\displaystyle\prod_{(v,v') \in E} 
\left( \frac{1}{n}\sum_{k=0}^{n-1} \omega^{(p(v')-p(v))k}\Pi_k
\right)_{\underline{\iota}(v), \, \underline{\iota}(v')}  \\
&=\frac{1}{n^{|E|}}\sum_{p \in [\![ 0,n-1 ]\!]^V }  \sum_{\varphi \in [ \![0,n-1 ]\!]^E }
\left( \prod_{ (v,v') \in E} \omega^{(p(v')-p(v)) \, \varphi(v,v')}
\left( \Pi_{\varphi(v,v')} \right)_{\underline{\iota}(v), \, \underline{\iota}(v')} \right)\\
&=\frac{1}{n^{|E|}} \sum_{\varphi \in [\![0,n-1]\!]^E}
\left(  \sum_{p  \in  [\![0,n-1]\!]^V}  \prod_{ (v,v') \in E}
\omega^{(p(v')-p(v)) \varphi(v,v')} \right) \left( \prod_{(v,v')\in E}
\left( \Pi_{\varphi(v,v')} \right)_{\underline{\iota}(v), \, \underline{\iota}(v')} \right) 
\end{align*}
where $[\![i,j]\!]^E$ denotes the set of maps from the edge set $E$ to $\{i,\ldots,j\}$.
We can rewrite the summation in the first parenthesis of the last equation as 
\be
\sum_{p  \in [\![0,n-1]\!]^V}  \prod_{ (v,v') \in E}  \omega^{(p(v')-p(v))
\varphi(v,v')} 
=  \sum_{p  \in [\![0,n-1]\!]^V} \prod_{v\in V}  \omega^{p(v) (O(v) -I(v))}
\label{eqn.s1} 
\ee
where
\bea
O(v) :=\sum_{\substack{e \in E  \\ e \textrm{ out of } v }}  \varphi(e)
\textrm{ and } I(v) : = \sum_{\substack{e \in E \\ e \textrm{ into } v }}
\varphi(e) \, .
\eea
Also, we have
\begin{align*}
 \sum_{p  \in [\![0,n-1]\!]^V} \prod_{v\in V}  \omega^{p(v) (O(v) -I(v))} 
&=   \prod_{v\in V }   ( 1 + \omega^{O(v) - I(v)} + \cdots
+  \omega^{(n-1)(O(v) - I(v))})  \\
& = \left \{ 
\begin{array}{ll}
n^{|V|} & \textrm{if } O(v) \equiv I(v) \textrm{ (mod $n$) for all } v \in V\\	
0& \textrm{otherwise}
\end{array} 
\right. \, .
\end{align*}
Namely, Equation~\eqref{eqn.s1} reduces to $n^{|V|}$ if $\varphi$ is a
$\BZ /n\BZ$-flow and vanishes, otherwise. Combining all the above, we obtain
\begin{align*}
W_{\calT^{(n)}}(G) &=  \frac{1}{ \sigma(G) \, n^{|E|-|V|} }
\sum_{\underline{\iota} \in [\![1,N]\!]^V}
\sum_{\varphi}\left(
\displaystyle\prod_{v \in V}  \Gamma_{\underline{\iota}(v)}^{(d(v))}
\prod_{ (v,v')\in E}  \left( \Pi_{\varphi(v,v')}
\right)_{ \underline{\iota}(v) \, \underline{\iota}(v') }  \right)\\
&=   \frac{1}{ \sigma(G) }
\sum_{\underline{\iota} \in  [\![1,N]\!]^V}
 W^{(n)}_{\calT} (G;\underline{\iota}) =W^{(n)}_{\calT}(G)\,  .
\end{align*}
Here we use the fact that the first betti
number of a connected  graph $G$ is given by  $|E(G)|- |V(G)|+1$. 
This completes the proof of Theorem~\ref{thm.feynman}
\qed


\section{Determining the shape of the loop invariants of cyclic covers}
\label{sec.riemann}

In this section we give proofs of theorem stated in Introduction.
Two main ingredients  are Theorem~\ref{thm.feynman}, which expresses the $\ell$-loop
invariants $\Phi_{\calT^{(n)} \!,\, \ell}$ of cyclic covers in terms of sums over
$\BZ/n\BZ$-tori of the inverse of a product of evaluations of $\delta_\calT(t)$ at
monomials, and Lemma~\ref{lem.comp} below, which
converts a sum over a $\BZ/n\BZ$-torus into a short rational function. It is inspired
by, and closely related to, the problem of counting lattice points in rational convex
polyhedra, and to the problem of expressing lattice point generating series in terms
of short rational functions~\cite{Barvinok,Barvinok-Woods}.

To phrase it, we fix a vector $c=(c_1,\dots,c_s)$ of nonzero complex numbers other
than the complex roots of unity and let $\BQ[c,1/S]$ denote the ring where $S$
is the set of $1-\prod_{i=1}^s c_i^{n_i}$ for all integers $n_i$ satisfying
$\prod_{i=1}^s c_i^{n_i} \neq 1$.

\begin{lemma}
\label{lem.comp}
Let $T_0,\ldots, T_s$ be Laurent monomials in variables $t_1,\ldots,t_d$.
If $T_i = t_i$ for $i=1,\ldots, d$ (hence $s \geq d$) and the exponents
of $T_{d+1},\ldots, T_s$ are in $\{0,\pm 1\}$, then there exists a polynomial
\bea
p(x_1,\dots,x_s,y) \in \BQ[c,1/S][x_1,\ldots,x_s,y]
\eea
of $x_i$-degree at most 1 and  $y$-degree at most $s-d$ such that 
\begin{equation}
\label{eqn.lemma}
\sum_{t_1^n = \cdots =t_d^n=1}  \frac{T_0}{(1-c_1\, T_1) \cdots (1-c_s\,T_s) } = 
n^d p\left(\frac{1}{1-c_1^n},\dots,\frac{1}{1-c_s^n},n\right)
\end{equation}
for all but finitely many $n$. 
\end{lemma}

\begin{proof}
For the sake of exposition we first consider the case $T_0=1$.
Since we take the sum over $t_1^n=\cdots =t_d^n=1$, the left-hand side
of~\eqref{eqn.lemma} equals to
\small{
\begin{align}
& \frac{1}{\prod_i(1-c_i^n)} \sum_{t_1^n = \cdots =t_d^n=1 } 
\prod_{i=1}^s \frac{1-c_i^nT_i^n}{1-c_i\, T_i} \nonumber \\
&=  \frac{1}{\prod_i(1-c_i^n)} \sum_{t_1^n = \cdots =t_d^n=1}
\prod_{i=1}^s  \left( 1 + c_i\,T_i +\cdots (c_i\,T_i) ^{n-1} \right)   \nonumber  \\
&=  \frac{1}{\prod_i(1-c_i^n)} \sum_{ k_1,\ldots,k_s = 0}^{n-1}
\sum_{t_1^n = \cdots =t_d^n=1} c_1^{k_1} \dots c_s^{k_s} \, T_1^{k_1}  \cdots  T_s^{k_s} .
\label{eqn.expand}
\end{align}
}
On the other hand, for any Laurent monomial $T$ in $t_1,\ldots,t_d$, we have
\bea
\sum_{t_1^n=\cdots=t_d^n=1} T = 
\left\{
\begin{array}{ll}
n^d & \textrm{if the exponents of $T$ are multiple of $n$} \\
0 & \textrm{otherwise}
\end{array} 
\right. \, .
\eea
From the fact that $T_i = t_i$ for $i=1,\ldots, d$, there are exactly $n^{s-d}$
monomials in~\eqref{eqn.expand} whose exponents are multiple of $n$.
Indeed, if we choose $0 \leq k_{d+1},\ldots, k_s < n$ freely, there is unique
$0 \leq k_1,\ldots, k_d <n$  such that  the exponents of $T_1^{k_1} \cdots
T_s^{k_s}$ are multiple of $n$. Thus we may  write Equation~\eqref{eqn.expand}
as
\be
\label{eqn.sum2}
\frac{n^d}{\prod_i(1-c_i^n)} \sum_{ k_{d+1},\ldots,k_s = 0}^{n-1}
c_1^{k_1} \dots c_s^{k_s} \, ,
\ee
regarding $k_1,\ldots,k_d$  as functions in $k_{d+1},\ldots, k_s$. Precisely, these
functions are given as
\bea
k_i = [R_i(k_{d+1},\ldots,k_s)]_n, \quad i=1,\ldots,d
\eea
where $R_i(k_{d+1},\ldots,k_s)$ is a linear combination of $k_{d+1}, \ldots, k_s$
with coefficients in  $\{0,\pm 1\}$ and $0 \leq  [x]_n < n$ denotes the integer
congruent to $x$ in modulo $n$.
	
Let $P_n = \{(k_{d+1},\ldots, k_s) \in \BZ^{s-d} \,\,|\,\,
0 \leq k_{d+1} ,\ldots, k_s \leq n-1 \}$ 
and partition $P_n$ with respect to $r=(r_1,\ldots, r_d) \in \BZ^d$ by letting
\bea
P_{n,r} = \{(k_{d+1},\ldots, k_s) \in P \,\,|\,\,
r_i n \leq R_i(k_{d+1},\ldots,k_s) \leq (r_i+1)n -1, \ i=1,\ldots,d\} \, .
\eea
Note that $P_{n,r}$ is empty for all but finitely many $r \in \BZ^d$ and  is an
integral polytope, since the coefficients of $R_1,\ldots,R_d$ are  in $\{0,\pm1\}$.
Also, we have
\begin{align}
\sum_{(k_{d+1},\ldots,k_{s} ) \in P_{n,r}} c_1^{k_1}  \cdots c_s^{k_{s}} &= 
c_1^{-n r_1} \cdots c_d^{-n r_d}
\sum_{(k_{d+1},\ldots,k_{s} ) \in P_{n,r}} c_1^{R_1}  \cdots c_d^{R_d} c_{d+1}^ {k_{d+1}}  
\cdots  c_s^{k_{s}}  \nonumber \\
&=c_1^{-n r_1} \cdots c_d^{-n r_d} \sum_{(k_{d+1},\ldots,k_{s})\in P_{n,r}}  
q_{d+1}^ {k_{d+1}}  \cdots  q_s^{k_{s}}
\label{eqn.cs}
\end{align}
for some $q_{i}  \in \BC$ given by a product of 
$c_i$ and some of $c_1^{\pm1},\ldots,c_d^{\pm1}$.
On the other hand, the formula of Brion (see e.g~\cite{Barvinok-Pommersheim}) allows
us to compute the generating function of $P_{n,r}$ 
\be
\label{eqn.gen}
F(P_{n,r}) = \sum_{(k_{d+1},\ldots,k_{s})\in P_{n,r}}
t_{d+1}^{k_{d+1}} \cdots  t_{s}^{k_{s}} 
\ee
in terms of the cones associated to the 
vertices of $P_{n,r}$. If we translate these cones so that their vertices become the
origin, they depend only on the coefficients of $R_1,\ldots,R_d$; in particular,
does not depend on $n$. We thus obtain from the formula of Brion that
\be
\label{eqn.gen2}
F(P_{n,r}) =\sum_{v:\textrm{vertex of } P_{n,r}}  f_v(t_{d+1},\ldots,t_s) \,
g_v(t_{d+1}^n,\ldots,t_{s}^n)
\ee
where $f_v$ is a rational function (not depending on $n$) and $g_v$ is a Laurent
polynomial (coming from the vertex translation). The denominator of $f_v$ is of the
form $\prod_i (1-t_{d+1}^{a_{i,d+1}} \dots t_{s}^{a_{i,s}})$ for integers $a_{i,j}$
coming from the rays of the cones of $P_{n,r}$ at its vertices. 

Substituting $t_{d+1} = q_{d+1},\ldots, t_s=q_s$
to Equations~\eqref{eqn.gen} and~\eqref{eqn.gen2}, we deduce that~\eqref{eqn.cs} 
is a polynomial in $c_1^{\pm n},\ldots,c_s^{\pm n}$ and $n$ with coefficients in 
$\BQ[c,1/S]$ where the the exponent of $n$ is at most  dimension of $P_{n,r}$,
hence at most $s-d$. This proves that \eqref{eqn.sum2} is also a polynomial in
$c_1^{\pm n},\ldots,c_s^{\pm n}$ and $n$ with coefficients in $\BQ[c,1/S]$, i.e.,
\be
\label{eqn.sum3}
\frac{n^d}{\prod_i(1-c_i^n)} \sum_{ k_{d+1},\ldots,k_s = 0}^{n-1}
c_1^{k_1} \dots c_s^{k_s} 
= \frac{n^d q(c_1^n,\ldots,c_s^n,n)}{\prod_i(1-c_i^n)}
\ee
for some polynomial
$q(x_1,\ldots,x_s,y) \in \BQ[c,1/S][x_1^{\pm1},\ldots,x_s^{\pm1},y]$ with $y$-degree 
at most $s-d$. Comparing the degree of the above equation with respect to $c_i^n$, 
we deduce that $q$ is a polynomial in $x_i$ with $x_i$-degree at most $1$.
Then rewriting the right-hand side of~\eqref{eqn.sum3} as a polynomial in
$1/(1-c_i^n)$, we obtain a desired polynomial $p$ satisfying~\eqref{eqn.lemma}.
This completes the proof when $T_0=1$.
When $T_0 \neq 1$, the same argument holds for large $n$ such that the exponents
of $T_0$ are in between $-n$ and $n$.
\end{proof}

The following example illustrates the above lemma.

\begin{example}
\label{ex.lemma}
When $s=1$, we have
\be
\label{av1}
\sum_{t^n=1}\frac{1}{1-ct}=\frac{n}{1-c^n} \,.
\ee
Using the above identity and the partial fraction decomposition 
\bea
\frac{1}{(1-at)(1-bt)}= \frac{1}{1-b/a} \frac{1}{1-at} + \frac{1}{1-a/b} \frac{1}{1-bt} 
\qquad a \neq b
\eea
of $1/((1-at)(1-bt))$ with respect to $t$, we obtain that
\bea
\sum_{t^n=1}\frac{1}{(1-at)(1-bt)}=
\frac{1}{1-b/a} \frac{1}{1-a^n} + \frac{1}{1-a/b} \frac{1}{1-b^n}
\qquad a \neq b\,.
\eea
Taking the limit of the above equality when $b$ tends to $a$, we obtain that
\be
\label{av2}
\sum_{t^n=1}\frac{1}{(1-at)^2}=
\frac{n^2}{(1-a^n)^2} + \frac{n-n^2}{1-a^n} \,.
\ee
Likewise, we have


\begin{small}
\be
\label{av3}
\begin{aligned}
\sum_{t^n=1}\frac{1}{(1-at)^3} &= \frac{n^3}{(1-a^n)^3}
-\frac{3(n^3-n^2)}{2(1-a^n)^2} +\frac{n^3-3n^2+2n}{2(1-a^n)} \, ,
\\
\sum_{t^n=1}\frac{1}{(1-at)^4} &= \frac{n^4}{(1-a^n)^4} +
\frac{2(-n^4+n^3)}{(1-a^n)^3} + \frac{7n^4 -18 n^3 + 11 n^2}{6(1-a^n)^2}-
\frac{n^4-6n^3+11n^2-6n}{6(1-a^n)} \,.
\end{aligned}
\ee
\end{small}
\end{example}

\subsection{Proof of Theorem~\ref{thm.phi}}

Recall Theorem~\ref{thm.feynman} that the $\ell$-loop invariant
$\Phi_{\calT^{(n)},\ell}$ is given as a
finite sum over $\ell$-loop Feynman diagrams. The weight of each Feynman
diagram is a sum over $\BZ /n\BZ$-flows of a  product of entries $\Gamma_i^{(k)}$,
which are $\BQ$-linear combinations of shape parameters and their inverses, times
entries of the propagator matrix $\Pi(\omega)$ where $\omega$ is a complex $n$-th
root of unity. 

We claim that the entries of $\Pi(t)$ are in $\delta_\calT(t)^{-1} F[t^{\pm1}]$. 
Clearly, entries of $\mathbf{B}(t)^{-1} \mathbf{A}(t)$ are in $\det
\mathbf{B}(t)^{-1}\BZ[t^{\pm1}]$. Recall that $\det \mathbf{B}(t)$ has a factor 
$t-1$ and that  $\mathbf{B}(t)^{-1} \mathbf{A}(t)$ 
equals to $\mathbf{B}_\lambda^{-1} \mathbf{A}_\lambda$ at $t=1$ 
(see Remark~\ref{rmk.lambda}). Thus, in fact, entries of $\mathbf{B}(t)^{-1}
\mathbf{A}(t)$ are in $\frac{t-1}{\det \mathbf{B}(t)} \BZ[t^{\pm1}]$.
Combining this with the definition of the twisted 1-loop invariant
$\delta_\calT(t)$, namely,
\bea
(t-1)\delta_\calT(t) = f  \det(\mathbf{A}(t) - \mathbf{B}(t)\Delta_{z'})
= f \det \mathbf{B}(t) \det \Pi(t)^{-1}
\eea
for some $f \in F$, we deduce that entries of $\Pi(t)$ are in
$\delta_\calT(t)^{-1} F[t^{\pm1}]$. 

Fix an $\ell$-loop Feynman diagram $G$. Recall that the set of all
$\BZ/n\BZ$-flows on $G$ is an abelian group isomorphic to $(\BZ/n \BZ)^{d}$
where $d$ is the first betti number of $G$. This can be proved by choosing a 
spanning tree of $G$, assigning an arbitrary element of $\BZ/ n\BZ$ to each 
edge of $G$ not in the tree,  and then extending this assignment in a unique way
to a flow on $G$. It follows that if we denote by  $t_1,\ldots,t_d$ the values of a 
flow on the edges not in the tree, then the values of a flow on the edges of $G$ 
are Laurent monomials in $t_1,\dots,t_d$ with exponents $0$, $1$ or $-1$.  Then
Theorem~\ref{thm.feynman} together with the above claim shows  that the 
contribution of $G$ to the  $\ell$-loop invariant $\Phi_{\calT^{(n)},\ell}$ is a linear
combination of 
\be
\label{eqn.contribution2}
\sum_{t_1^n = \cdots =t_d^n=1} \frac{T_0} {\delta_\calT(T_1) \cdots \delta_\calT(T_{s})}
\ee
where $s$ is the number of edges of $G$ and $T_i$ are Laurent monomials
in $t_1,\ldots,t_d$ satisfying the condition of Lemma~\ref{lem.comp}. 
Note that $s \leq 3 \ell -3$, as $\ell$-loop Feynman diagrams have at most
$3\ell-3$ edges.

\begin{remark}
\label{rem.phi} 
Applying Lemma~\ref{lem.comp} directly to~\eqref{eqn.contribution2}, we obtain the
existence of a polynomial $p_{\calT,\ell} \in  \calF_{(3 \ell-3)r} E[x_1,\ldots,x_r][y]$
satisfying Equation~\eqref{philn}. The non-resonance assumption is not required here, 
but the degree bound for $p_{\calT,\ell}$ may not be optimal.
\end{remark}  

From the non-resonance assumption, we have $\lambda_i^{\pm1} \neq
\lambda_j^{\pm1}$ if $i \neq j$. Therefore, by using the partial fraction decomposition,
we can write $1/\delta_\calT(t)$ as an $E[t^{\pm1}]$-linear combination of
 $1/(1-\lambda_i^{\pm1} t)$ for $i=1,\ldots,r$ so that~\eqref{eqn.contribution2} is given
as an $E[t_1^{\pm1},\ldots,t_d^{\pm1}]$-linear combination of terms of the form
\be
\label{eqn.Tsum}
\sum_{t_1^n = \cdots =t_d^n=1}  \frac{1}{(1-\lambda_{i_1}^{\pm1} T_1) \cdots
(1-\lambda_{i_s}^{\pm1} T_s)} \, .
 \ee
 
\begin{lemma}
\label{lem.twoell}
Let $S$ be a subset of the edge set of $G$ such that for each odd-degree vertex
of $G$, not all adjacent edges are contained in $S$. Then $ |S| \leq 2 \ell -2 $
and $|S| - (d-1) \leq \ell-1 $.
\end{lemma}
 
\begin{proof}
Let $n_k$ be the number of $k$-vertices in $G$. Note that  $ \# \textrm{(edges)} 
= \frac{1}{2} \sum_{k\geq1} k\, n_k$ and that $G \in \mathcal{G}_\ell$ implies 
$n_1 + n_2 +  \# \textrm{(loops)}\leq \ell$.
From the condition on $S$, we have
\begin{align}
|S| \leq \# \textrm{(edges)} -  \frac{1}{2} \sum_{k : \textrm{odd}} n_k
= n_2 + n_3+ 2 n_4 + 2 n_5 + \cdots \,.
\label{eqn.H}
\end{align}
On the other hand, the connectedness of $G$ implies that
\begin{align*}
\# \textrm{(loops)} - 1 = \# \textrm{(edges)} - \# \textrm{(vertices)} 
=   \sum_{k \geq 1} \frac{(k-2)}{2} n_k
\end{align*}
and thus
\be
\# \textrm{(loops)} -1 + \frac{1}{2} n_1  =  \frac{1}{2} n_3 + n_4
+ \frac{3}{2} n_5 + \cdots \, .
\label{eqn.H2}
\ee
Comparing the coefficients of $n_k$ for $k \geq 3$ in~\eqref{eqn.H} 
and~\eqref{eqn.H2}, we have
\bea
|S| \leq  n_1+ n_2 + 2\# \textrm{(loops)} -2 \leq 2 \ell-2 
\eea
and
\bea
|S|-(d-1)=|S|-\#( \textrm{loops})+1 \leq  n_1 + n_2
+ \#( \textrm{loops}) -1 \leq \ell -1 \, .
\eea
\end{proof}
 
Lemma~\ref{lem.twoell} implies that if $s > 2 \ell -2$, then there exists an
odd-degree vertex in $G$. Let $T_1,\ldots,T_{2m+1}$ be the values of a flow on
the adjacent edges of that vertex. From the definition of flow, the product
of some of $T_1,\ldots,T_{2m+1}$ should be equal to the product of 
the others. Namely, we have (after re-labeling $T_j$)
\bea
T_1 \cdots T_k = T_{k+1} \cdots T_{2m+1} \, 
\eea
for some $1 \leq k \leq 2m+1$, hence
\bea
\prod_{j=1}^k (\lambda_{i_j}^{-1}(\lambda_{i_j} T_j-1) +\lambda_{i_j}^{-1}) =
\prod_{j=k+1}^{2m+1}  (\lambda_{i_j}^{-1}(\lambda_{i_j} T_j-1) +\lambda_{i_j}^{-1}) \,.
\eea
Expanding the above equation, the constant term $\prod_{j=1}^k \lambda_{i_j}^{-1}-
\prod_{j=k+1}^{2m+1} \lambda_{i_j}^{-1}$ which is non-zero due to the non-resonance
assumption, is given by
as a linear combination of $1-\lambda_{i_1}T_1,\ldots, 1- \lambda_{i_{2m+1}}T_{2m+1}$
(and their  products). It follows that we can write 
\bea 
\frac{1}{(1-\lambda_{i_1}T_1) \cdots (1-\lambda_{i_{2m+1}} T_{2m+1})} 
\eea
as a linear combination of 
\bea
\frac{1}{(1-\lambda_{i_1} T_1) \cdots  
\widehat{(1- \lambda_{i_j} T_{j})}\cdots  (1-\lambda_{i_{2m+1}}T_{2m+1})},
\quad j=1,\ldots,{2m+1}
\eea
where the hat means that the $j$-th factor is excluded. This explains that we can 
decompose~\eqref{eqn.Tsum} into terms of 
the same form with $s \leq 2 \ell -2$. Then the existence of a polynomial
\be
\label{px2}
p_{\calT,\ell}(x_1,\dots,x_r,y) \in \calF_{2 \ell -2}
E[x_1,\dots,x_r][y]
\ee
satisfying Equation~\eqref{philn} follows from Lemma~\ref{lem.comp} where
the last inequality of Lemma~\ref{lem.twoell} shows that the $y$-degree of
$p_{\calT,\ell}$ is at most $\ell-1$ (and at least $1$; see  Equations~\eqref{eqn.newdef1}~ and \eqref{eqn.lemma}).
\qed
 
\begin{remark}
\label{rem.phi2}
If one uses Neumann-Zagier datum  with respect the meridian, then we need
to consider the contribution~\eqref{eqn.contribution2} coming from all subgraphs
of $\ell$-loop Feynman diagrams. This does not affect on the existence of the
polynomial $p_{\calT,\ell}$, but it may be a Laurent polynomial in the variable $y$.
\end{remark}
 
\begin{remark}
\label{rem.integrality}
The above proof gives an integrality statement for the coefficients of
$p_{\calT,\ell}$, namely we can replace the field $E$ in~\eqref{px2} by the ring
$\frac{1}{d_{\ell}} O_{F,S}$ where $O_F$ denotes the ring of integers of invariant
trace field $F$, $S$ denotes the set of all nonzero numbers of the form
$\prod_i \lambda_i^{n_i}-1$, as well as the denominators of the shapes $z$, $z'$
and $z''$ and $O_{F,S}=O_F[1/S]$ denotes the localization of $O_F$ with respect to
$S$. Finally, $d_\ell$ is a universal denominator that comes from the greatest
common divisor of the inverse automorphism factor of the $\ell$-loop Feynman diagrams
given explicitly in~\cite[Sec.9]{GZ:kashaev}.
\end{remark}


\subsection{Proof of Theorem~\ref{thm.quad}}

As a qudratic $\delta_\calT(t)$ automatically satisfies the non-resonance condition,
we obtain from Theorem~\ref{thm.phi} a Laurent polynomial
\bea
p_{\calT,\ell}(x,y) \in  \calF_{2 \ell-2}E[x][y]
\eea
such that for all but finitely many $n$, we have
\be
\label{eqn.des}
\Phi_{\calT^{(n)},\ell}=  p_{\calT,\ell} \left(\frac{1}{1-\lambda^n},n\right) 
\ee
where $\lambda$ is a root of $\delta_{\calT}(t)$.

\begin{lemma}
\label{lem.delta}
For a given $k \geq 1$ there exist polynomials
$\alpha_{k,i}(x) \in \BQ[\lambda,\frac{1}{\lambda^2-1}][x]$
of degree at most $k$ for $i=0, \ldots, k$ such that
\be
\label{eqn.deltasum} 
\sum_{t^n=1} \dfrac{1}{\delta(t)^k} = \alpha_{k,0}(n) +
\frac{\alpha_{k,1}(n)}{1-\lambda^n} +  \dots + \frac{\alpha_{k,k}(n)}{(1-\lambda^n)^k}
\ee
for all $n \geq 1$. In addition, $\alpha_{k,k}(x)= 2 \lambda^k(\lambda^2-1)^{-k}x^k$. 
\end{lemma}

\begin{proof}
For $k=1$ one easily checks that
\bea
\sum_{t^n=1} \frac{1}{\delta(t)} =
- \frac{\lambda n }{\lambda^2-1} +\frac{2 \lambda n}{(\lambda^2-1)(1-\lambda^n)} \, . 
\eea
Suppose that there are polynomials $\alpha_{k,0}(x), \ldots, \alpha_{k,k} (x)$
with $\deg \alpha_{k,i} \leq k$ satisfying the equation~\eqref{eqn.deltasum}. We
then take the derivative both sides of~\eqref{eqn.deltasum} with respect to $\lambda$.
From the left-hand side, we obtain
\begin{align*}
\frac{d}{dc} \left(\sum_{t^n=1} \frac{1}{\delta(t)^k}  \right) &
= \sum_{t^n=1} \left(\frac{k \, t^k}{(t-\lambda)^{k+1}(t-\lambda^{-1})^k}
-  \frac{k\, t^k \lambda^{-2}}{(t-\lambda)^k (t-\lambda^{-1})^{k+1}} \right) \\
&= \frac{k}{1-\lambda^{-2}}\sum_{t^n=1} \frac{1}{\delta(t)^{k+1}} \,. 
\end{align*}
From the right-hand side, we obtain
\begin{align*}
&\sum_{i=0}^k \frac{\frac{d}{d\lambda}\alpha_{k,i}(n)}{(1-\lambda^n)^i}
+ \sum_{i=1}^k \frac{i\, n \lambda^{n-1} \alpha_{k,i}(n)}{(1-\lambda^n)^{i+1}} \\
&=\sum_{i=0}^k \frac{\frac{d}{d\lambda} \alpha_{k,i}(n)}{(1-\lambda^n)^i}
+\sum_{i=1}^k \left(\frac{i\, n \lambda^{-1} \alpha_{k,i}(n)  }{(1-\lambda^n)^{i+1}}
-\frac{i\, n \lambda^{-1} \alpha_{k,i}(n)}{(1-\lambda^n)^{i}} \right)\,.
\end{align*}
Comparing the above two equations, we obtain polynomials
$\alpha_{k+1,0}(x),\ldots,\alpha_{k+1,k+1}(x)$ satisfying
\bea
\sum_{t^n=1} \frac{1}{\delta(t)^{k+1}} =
\alpha_{k+1,0}(n) + \frac{\alpha_{k+1,1}(n)}{1-\lambda^n} + 
\cdots  +  \frac{\alpha_{k+1,k+1}(n)}{(1-\lambda^n)^{k+1}} \, .
\eea
In particular, $\alpha_{k+1,k+1}(x) =x \alpha_{k,k}(x)/(\lambda-\lambda^{-1})$ and
$\deg \alpha_{k+1,i} \leq {k+1}$. This completes the proof.
\end{proof}

\begin{lemma}
\label{lem.last}
For a given $k \geq 1$ there exist polynomials
$\beta_{k,i}(x) \in \BQ[\lambda^{\pm1}, \frac{1}{\lambda^2-1}][x]$ of degree at most
$k$ for $i=0,\dots,k$ satisfying
\bea
\frac{1}{(1-\lambda^n)^k} = \sum_{t^n=1} \left( \beta_{k,0}(n^{-1}) +
\frac{\beta_{k,1}(n^{-1})}{\delta(t)} + \dots + \frac{\beta_{k,k}(n^{-1})}{\delta(t)^k}
\right)
\eea
for all $n \geq 1$.
\end{lemma}

\begin{proof}
Writing Lemma~\ref{lem.delta} as a linear combination of
$1, (1-\lambda^n)^{-1}, \ldots, (1-\lambda^n)^{-k}$, we obtain a $(k+1) \times (k+1)$
matrix $\alpha(x)$ with entries in $\BQ[\lambda,\frac{1}{\lambda^2-1}][x]$ such that
\be
\label{eqn.B}
\alpha(n)
\begin{pmatrix}
1\\
(1-\lambda^n)^{-1} \\
\vdots \\
(1-\lambda^n)^{-k}
\end{pmatrix} = \begin{pmatrix}
1\\
\sum_{t^n=1} \delta(t)^{-1} \\
\vdots \\
\sum_{t^n=1} \delta(t)^{-k} \\
\end{pmatrix} \, .
\ee
Moreover, $\alpha(x)$ is a lower-triangular matrix whose $i$-th row consists 
of polynomials of degree at most $i$ and whose diagonal is
$(2\lambda^k(\lambda^2-1)^{-k} x^k,\ldots,2\lambda(\lambda^2-1)^{-1} x,1)$.
It follows that entries of the 
inverse $\alpha(x)^{-1}$ are in $ \BQ[\lambda^{\pm1}, \frac{1}{\lambda^2-1}][x^{-1}]$.
We then obtain the lemma by multiplying $\alpha(n)^{-1}$ on both sides of 
Equation~\eqref{eqn.B}.
\end{proof}

Applying Lemma~\ref{lem.last} to Equation~\eqref{eqn.des}, we obtain the
existence of Laurent polynomials 
$q_{\calT,\ell}(x,y) \in \calF_{2\ell-2}E[x][y^{\pm 1}]$ satisfying
for all but finitely many $n$
\be
\label{pjthm1a}
\Phi_{\calT^{(n)} \!, \ell}=
\sum_{t^n=1} q_{\calT,\ell} \left(\frac{1}{\delta_{\calT}(t)}, \frac{1}{n} \right)  \, .
\ee
From the fact that the asymptotic of $\Phi_{\calT^{(n)},\ell}$ is a linear in $n$
(see Theorem~\ref{thm.psi}), we deduce that  $q_{\calT,\ell}$ is a polynomial in $y$.
It remains to show that the coefficients of $q_{\calT,\ell}$ lie in the field
$F(\lambda+\lambda^{-1})=F$. This follows from the fact that the left-hand side
of~\eqref{pjthm1a} is invariant under $\lambda \mapsto \lambda^{-1}$, and hence
so is the right-hand side. This completes the proof of Theorem~\ref{thm.quad}.

\begin{remark}
\label{rem.quad}
Theorem~\ref{thm.quad} requires only that $\delta_\calT(t)$ has no complex roots of
unity as roots.
\end{remark}  

\subsection{Proof of Theorem~\ref{thm.psi}}
\label{sec.psi}

In this section we give a proof of our main Theorem~\ref{thm.psi} by combining
Theorem~\ref{thm.feynman} and Theorem~\ref{thm.phi}. 

Theorem~\ref{thm.feynman} expresses the $\ell$-loop invariant
$\Phi^\conn_{\calT^{(n)},\ell}$ as a Riemann sum, which is asymptotic to a Riemann
integral as $n$ tends to infinity. This makes sense since the function to be summed
or integrated is a sum of products of $1/\delta_\calT(t)$ and the latter is nonzero
on the unit circle by assumption. 
The corrections to the approximation of the Riemann sum by a Riemann integral
are given by the Euler-Maclaurin summation formula and vanish to all orders in $1/n$,
since the functions to be integrated are periodic. This proves the existence of
$\Psi_{\calT,\ell}^\conn$ given by sums of multidimensional integrals over tori
satisfying Equation~\eqref{eqn.lloop3} up to $O(1/n^\infty)$. Since
$(\calT^{(n)})^{(m)}=\calT^{(nm)}$ for all integers $m$ and $n \geq 1$, 
Equation~\eqref{eqn.lloop3} implies that $\Psi_{\calT,\ell}^\conn$ are multiplicative
under cyclic covers, i.e., satisfy Equation~\eqref{psin}.

It remains to improve the estimate $O(1/n^\infty)$ in Equation~\eqref{eqn.lloop3}
to a sharp exponential estimate. For that, we use the explicit shape of the Riemann
sums given in Theorem~\ref{thm.phi}. We can choose roots $\lambda_1,\ldots,
 \lambda_r$ of $\delta_\calT(t)$ so that they are strictly inside the unit disk
(since they come in pairs $\lambda,1/\lambda$ and by assumption, none is on the
unit circle). Then the difference between the $n$-Riemann sum and the Riemann
integral is $O(p(n)|\lambda_j|^{n})$ as $n$ tends to infinity where $\lambda_j$ is
the smallest, in absolute value, root of $\delta_\calT(t)$ and $p(n)$
is a polynomial of $n$ of bounded degree. 

We finally give a Feynman diagram definition of $\Psi_{\calT, \ell}^\conn$
using $S^1$-flows. Our definition is analogous to the $\BZ/n\BZ$ flows used in
Section~\ref{sub.flows} and in
Theorem~\ref{thm.feynman} to describe the connected $\ell$-loop invariants
of $n$-cyclic covers. An $S^1$-\emph{flow} on a Feynman diagram $G$ is a map 
$\varphi : E(G) \rightarrow S^1$ such that  for all $v \in V(G)$
\be
\prod_{\substack{e \in E(G) \\ e \textrm{ into } v }} \varphi (e)
= \prod_{\substack{e \in E(G) \\ e \textrm{ out of } v }}  \varphi(e) \, .
\ee
The set of $S^1$-flows on $G$ is isomorphic to (as a 
multiplicative set) the $d$-dimensional torus $(S^1)^d$ where $d$ is the 
first betti number of $G$. In addition, the value $T_e$ of an $S^1$-flow on 
each edge of $G$ is a  Laurent monomial in $d$ variables, say
$t_1,\ldots,t_d$.
For a vertex-labeling $\iota : V(G)\rightarrow \{1,\ldots,N\}$  we define
\be
W^{(\infty)}_{\calT}(G;\iota) := \prod_{v\in V(G)} \Gamma^{(d(v))}_{\iota(v)}
\int_{(S^1)^d}  \prod_{(v,v)' \in E(G)} \Pi(T_{(v,v')})_{\iota(v), \, \iota(v')}\, 
\frac{d t_1 \cdots d {t_d}}{t_1 \cdots t_d} 
\ee
and let 
\be
W^{(\infty)}_{\calT}(G) := \frac{1}{\sigma(G)}
\sum_{\iota} W^{(\infty)}_{\calT}(G;\iota) 
\ee
where the sum is over all vertex-labeling $\iota : V(G)\rightarrow \{1,\ldots,N\}$.
For instance, we have
\be
W^{(\infty)}_\calT(G) = \frac{1}{8}
\sum_{1 \leq i,j \leq N} \Gamma^{(3)}_i \Gamma^{(3)}_j 
\int_{(S^1)^2} \Pi(t_1)_{ij} \Pi(t_2)_{ij} \Pi(t_1^{-1} t_2^{-1})_{ij}
\frac{d t_1 d t_2}{t_1 t_2}
\ee
for a Feynman diagram given as in Figure~\ref{fig.diagram}
(cf. Equation~\eqref{eqn.example}). The leading term $\Psi_{\calT,\ell}^\conn$ is
given by
\be
\label{eqn.psi}
\Psi_{\calT\!,\ell}^\conn = 
\mathrm{coeff}
\left[ \sum_{G \in \mathcal{G}_\ell} W^{(\infty)}_{\calT}(G), \,\hbar^{\ell-1}
\right]+ \Gamma^{(0)} \, .
\ee
Note that contrast to the $\ell$-loop invariant $\Phi_{\calT\!,\ell}^\conn$,
a choice of a peripheral curve is not used in the above formula for
$\Psi_{\calT\!,\ell}^\conn$. This completes the proof of Theorem~\ref{thm.psi}. 
\qed

\begin{remark}
\label{rem.psi}
Theorem~\ref{thm.psi} holds under the assumption that $\delta_{\calT}(t)$ has no
roots on the unit circle,
otherwise the limit does not exist. Also, the exponentially small bound is optimal.
Compare with the following toy example when $\lambda$ is not a complex root of unity
\begin{equation*}
\sum_{t^n=1} \frac{t}{(t-\lambda)(t-\lambda^{-1})} =
\frac{n}{\lambda-\lambda^{-1}} \Big(\frac{1}{1-\lambda^n}-\frac{1}{1-\lambda^{-n}}\Big)
= \begin{cases}
  n \frac{\lambda}{\lambda^2-1} + O(n \lambda^{n}) & \text{if }  |\lambda|<1 \\[3pt]
  n \frac{\lambda}{1-\lambda^2}  + O(n \lambda^{-n}) & \text{if } |\lambda|>1
\end{cases}
\end{equation*}
as $n$ tends to infinity, whereas the limit does not exist if $|\lambda|=1$. 
\end{remark}

\subsection{Generalized power sums}
\label{sub.psums}

In this section we review briefly the \emph{generalized power sums} and their
properties, following~\cite{Poorten, Everest}. The latter are sequences of the form
\be
\label{expsums}
a_n = \sum_{j=1}^m A_j(n) \lambda_j^n
\ee
with roots $\lambda_j$ for $1 \leq j \leq m$ distinct complex numbers and coefficients
$A_j(n)$ polynomials of degree $d_j-1$ for positive integers $d_j$ for $1 \leq j
\leq m$. The order of the generalized power sum is $d=\sum_{j=1}^m d_j$. Generalized
power sums are solutions to linear recursions with constant coefficients,
explicitly,
\be
\label{anrec}
a_{n+d}= s_1 a_{n+d-1} + \dots + s_d a_n, \qquad n=0,1,2,\dots
\ee
where $s(t)=\prod_{j=1}^m (1-\lambda_j t)^{d_j}=1-s_1 t - \dots - s_d t^d$,
and their generating series
\bea
\sum_{n=0}^\infty a_n t^n = \frac{r(t)}{s(t)}
\eea
is a rational function of negative $t$-degree. Note that if the roots of a generalized
power sum $a(n)$ are known, then $(a_n)$ is determined by its first $d$ values, as
follows from recursion~\eqref{anrec}. 

A special but important example of
generalized power sums are the quasipolynomials, whose roots are complex roots
of unity. Quasipolynomials play a key role in the lattice point counting in
rational convex polyhedra~\cite{Ehrhart,BV1,Barvinok-Pommersheim,Beck-Robins}. 

Note that the vector space of generalized power sums is a ring with respect
to pointwise multiplication of sequences.

We next recall the Lech--Mahler--Skolem theorem, whose statement is elementary
and whose proof requires $p$-adic analysis.

\begin{theorem}
\label{thm.SML}\cite{Sk,Ma,Le}
The zero set $\{n \in \BN \, \, | a_n=0\}$ of a generalized power sum $(a_n)$
is the union of a finite set and a finite set of arithmetic progressions.
\end{theorem}
Moreover, if the roots of a generalized power sum are multiplicatively independent,
then its zero set is either finite or all the natural numbers. 

We now come to the proof of Proposition~\ref{prop.rational}.
By putting Equation~\eqref{px} of Theorem~\ref{thm.phi} in a common denominator
and abbreviating $p_{\calT,\ell}$ by $p$, it follows that
\be
\label{tip}
((1-\lambda_1^n)\dots (1-\lambda_r^n))^{2\ell-2}
\Phi_{ \mathcal{T}^{(n)},\ell} = \ti p(\lambda_1^n,\dots,\lambda_r^n,n)
\ee
where $\ti p(x_1,\dots,x_r,y) =
(1-x_1)^{2\ell-2} \dots (1-x_r)^{2\ell-2}
p\big(\tfrac{1}{1-x_1},\dots,\tfrac{1}{1-x_r},y) \in \calF_{r(2\ell-2)}E[x][y]$.
It is easy to see that the linear map 
$\calF_{2\ell-2}E[x][y] \to \calF_{r(2\ell-2)}E[x][y]$ that sends $p$ to $\ti p$
is injective and that the sequence
$(\ti p(\lambda_1^n,\dots,\lambda_r^n,n))$ is a generalized power sum. Since
the dimension of $\calF_sE[x]$ is $\binom{r+s}{s}$ and the $y$-degree of $\ti p(x,y)$
is at least 1 and at most $\ell-1$, the non-resonance assumptions of $\Lambda$
imply that the set of
roots of this generalized power sum are monomials in $\lambda_1,\dots,\lambda_r$
(in total at most $\binom{r+r(2\ell-2)}{r}$) of $y$-degree at least 1 and at
most $\ell-1$. It follows 
that the degree of the generalized power sum~\eqref{tip} is at most
$(\ell-1)\binom{r+r(2\ell-2)}{r}$. Thus, if $\Lambda$ is known, then the first
$(\ell-1)\binom{r+r(2\ell-2)}{r}$ many initial values of it determine it completely,
and moreover determine $\ti p$. Since the map $p \mapsto \ti p$ is injective,
the above discussion together with~\eqref{tip} imply that 
the polynomial $p_{\calT,\ell}$ is determined by $\Lambda$ and by 
$(\ell-1)\binom{r+r(2\ell-2)}{r}$ many initial values of $\Phi_{\calT^{(n)},\ell}$.

It remains to reduce the number of initial values of $\Phi_{\calT^{(n)},\ell}$
from $(\ell-1)\binom{r+r(2\ell-2)}{r}$ to $d_{\ell,r}:=(\ell-1)\binom{r+2\ell-2}{r}$.
This is possible, because $\ti p$ lies in a $d_{\ell,r}$-dimensional
subspace of $\calF_{r(2\ell-2)}E[x][y]$, but the details are more delicate. 

To prove this, we write $p(x,y)$ in terms of its monomials as
$p(x,y)=\sum_{(\alpha,\beta) \in \calC} c_{\alpha,\beta} x^\alpha y^\beta$
where $\alpha=(\alpha_1,\dots,\alpha_r)$, $\beta \in \BN$, $x=(x_1,\dots,x_r)$
and $x^\alpha=x_1^{\alpha_1} \dots x_r^{\alpha_r}$. Then,
$|\alpha|=\alpha_1+\dots+\alpha_r \leq 2\ell-2$, $\beta \leq \ell-1$
and $|\calC| = d_{\ell,r}$. 

Consider Equation~\eqref{philn} for $n=n,n+1,\dots,n+d_{\ell,r}-1$ as a system of
linear equations with unknowns $c_{\alpha,\beta}$. The corresponding square matrix
$\big(\tfrac{1}{(1-\lambda^{n+j})^\alpha} (n+j)^\beta\big)$ has rows indexed by
$j=0,\dots,d_{\ell,r}-1$ and columns indexed by $(\alpha,\beta) \in \calC$. 
After putting the matrix into a common denominator, its numerator is a generalized
power sum, whose roots are monomials in $\lambda_j$. If this generalized power sum
is not identically zero, the Lech-Mahler-Skolem theorem and the non-resonance
assumption on $\lambda_j$ implies that the zero set of the generalized
power sum is finite, hence we can solve the system of linear equations using
$d_{\ell,r}$ consecutive values of $\Phi_{\calT^{(n)},\ell}$ to recover the
coefficients of $p(x,y)$.

Thus, we need to prove that the generalized power sum is not identically zero.
Assume otherwise. Using the multiplicative independence of $\Lambda$, it follows that
that each polynomial $A_j(n)$ (in the notation of~\eqref{expsums}) is identically
zero and this implies that the determinant
$\Delta_\calC(y,\lambda,n):=\det(M_\calC(y,\lambda,n))$ of the matrix
$M_\calC(y,\lambda,n):=(\big(\tfrac{1}{(1-y \lambda^{j})^\alpha}
(n+j)^\beta\big))$ vanishes identically for all $y=(y_1,\dots,y_r)$ and all $n$. Let us 
totally order $\calC$ by $(\alpha,\beta)$ by $(\alpha,\beta) \geq (\alpha',\beta')$
if and only if $|\alpha| > |\alpha'|$ or $|\alpha|=|\alpha'|$ and $\beta \geq \beta'$.
We extend this to a partial order for subsets $\calC'$ of $\calC$ by
$\calC' \geq \calC''$ if the maximum element of $\calC'$ is greater than or equal
to the maximum element of $\calC''$. 

Let $\calS
=\{ \calC' \,\, | \Delta_{\calC'}(y,\lambda,n)=0 \,\, \text{for all} \,\,y,n \}$.
Note that $\calS$ is nonempty since it contains $\calC$. Let $\calC' \in \calS$
denote an element of $\calS$ with $|\calC'|$ minimum. Let $(\alpha',\beta')$ denote
the maximum element of $\calC'$. We distinguish two cases. 

\noindent
Case 1. If $\beta'>0$, then after doing column
operations on the matrix $M_{\calC'}(y,\lambda,n)$ we can assume that the $j=0$
row has vanishing entries for $(\alpha',\beta)$ except at $\beta=0$. Then, the
matrix obtained from $M_{\calC'}(y,\lambda,n)$ by removing the $(\alpha',0)$ row and
column is $M_{\calC'-\{\alpha',\beta'\}}(\lambda y,\lambda,n+1)$. 

Expanding the determinant $\Delta_{\calC'}(y,\lambda,n)$ with respect to the $0$-th row,
it follows that
\be
\label{deltac}
0=\Delta_{\calC'}(y,\lambda,n)=\tfrac{1}{(1-y)^{\alpha'}} 
\Delta_{\calC'-\{\alpha',\beta'\}}(\lambda y,\lambda,n+1) + (\text{other terms}) \,.
\ee
Both sides of the above equation are rational functions of $y$, and each term
of the determinants and of the other terms are products of
$\tfrac{1}{(1-y \lambda^j)^{\alpha''}}
=\prod_{i=1}^r \tfrac{1}{(1-y_i\lambda_i)^{\alpha''_i}}$. 
On the other hand, the other terms do not have a singularity $(1-y)^{\alpha'}$.
It follows that $\Delta_{\calC'-\{\alpha',\beta'\}}(\lambda y,\lambda,n+1)=0$. Thus,
$\calC'-\{\alpha',\beta'\} \in \calS$ but $|\calC'-\{\alpha',\beta'\}| < |\calC'|$
a contradiction. 

\noindent
Case 2. If $\beta'=0$, then expanding the determinant
$\Delta_{\calC'}(y,\lambda,n)$ with respect to the $0$-th row. Equation~\eqref{deltac}
still holds and the same reasoning as in the first case implies that
$\Delta_{\calC'-\{\alpha',\beta'\}}(\lambda y,\lambda,n+1)=0$ giving a contradiction
once again. 
 
This concludes the proof of the second part of the proposition.

For the first part, observe that the $n$-th cyclic resultant $N_n(\delta_\calT)$ of the
twisted 1-loop invariant  $\delta_\calT(t)$ is given by
\bea
N_n(\delta_\calT)=\prod_{\omega^n=1}\delta_\calT(\omega)
=\prod_{j=1}^r (1-\lambda_j^n)(1-\lambda_j^{-n}),
\eea
which equals to $(1-\lambda_1^n)^2 \dots (1-\lambda_r^n)^2$ times the $n$-th power
of a signed monomial. This and Equation~\eqref{tip} imply that the sequence
$(N_n(\delta_\calT)^{\ell-1}\Phi_{ \mathcal{T}^{(n)},\ell})$ is a generalized power
sum, hence its generating series~\eqref{phirat} is rational. This concludes the
proof of Proposition~\ref{prop.rational}.
\qed

\subsection{Asymptotics}
\label{sub.determines}

The rest of the section is devoted to the proof of Proposition~\ref{prop.determines}.
Fix a rational function $R(x,y) \in \BC(x)[y]$ regular at $x=0$, where
$x=(x_1,\dots,x_r)$. Then, we can consider the image of $R(x,y)$ in the completed
power series ring $\BC[[x]][y]$
\be
\label{Rxy}
R(x,y)=\sum_{k \in \BN^r} a_k(y) x^k
\ee
where $k=(k_1,\dots,k_r) \in \BN^r$ and $x^k=x_1^{k_1} \dots x_r^{k_r}$, and
where $a_k(y)$ are polynomials in $y$ of degree at most $d$. Fix a set
$\Lambda_+=\{\lambda_j \,\, | \,\, j=1,\dots r\}$ of multiplicatively independent
nonzero complex numbers with $|\lambda_j|<1$ for all $j=1,\dots,r$.
Consider the set $\calL=\{\lambda^k \, \, | \, \, k \in \BN^r\}$ 
where $\lambda^k=\lambda_1^{k_1} \dots \lambda_r^{k_r}$, and let $\calE$ denote
the set of the absolute values of the elements of $\calL$. Since $0<|\lambda_j| <1$
for all $j$, it follows that $\calE$ is a discrete subset of $(0,1)$ with $0$
its only limit point. Hence, we can write $\calE=\{L_m \,\, | \,\, m \in \BN\}$
where $1 > L_0 > L_1 > L_2 > \dots$. Using this, we can partition
$\calE = \sqcup_{m \in \BN} \calE_m$, where $\calE_m$ is the set of $\lambda^k$
with $|\lambda^k|=L_m$. The assumptions on $\{\lambda_j\}$ imply that each $\calE_m$
is a finite set and that the map
$\{x^k \,\, | \,\, k \in \BN\} \to \calL$ that sends $x^k$ to $\lambda^k$ is 1-1.
Thus, Equation~\eqref{Rxy} can be written in the form
\be
\label{Rxy2}
R(x,y)=\sum_{m=0}^\infty R_m(x,y), \qquad
R_m(x,y) = \sum_{k \in \calE_m} a_k(y) x^k 
\ee
where $R_m(x,y)$ is a finite sum. Letting $a_{m,n}=R_m(\lambda^n,n)$,
it follows that for each fixed $m$, the sequence $n \mapsto a_{m,n}$ is a
generalized power sum that satisfies $a_{m,n}=O(n^d L_m^n)$, that the series
\be
a_n = \sum_{m=0}^\infty a_{m,n}
\ee
is absolutely convergent and that the partial sums satisfy
\be
\label{amn}
a_n - \sum_{m=0}^{M-1} a_{m,n} = a_{M,n} + O(n^d L_{M+1}^n) \,.
\ee
It follows by induction on the natural number $M$ that the left hand side
of~\eqref{amn} is a generalized power sum which is a sequence of
Nilsson type~\cite{Ga:Nilsson}, whose asymptotic expansion to all orders in $1/n$
is terminating and given exactly by $a_{M,n}$. This implies by induction that
$a_n$ (or infinitely many evaluations of it), determine $a_{m,n}=R_m(\lambda^n,n)$
for all $m$. This implies in turn that $a_n$ determines the the polynomial $R_m(x,y)$
and the set $\calE_m$ for all $m$, and hence determines the image of $R(x,y)$
in $\BC[[x]][y]$ by~\eqref{Rxy} as well as the set $\calE$ (and hence $\Lambda_+$).
Since the map $\BC(x)[y] \to \BC[[x]][y]$ (partially defined on rational functions
which are regular at $x=0$) is injective, this completes the first part of
Proposition~\ref{prop.determines}.

For the second part, assume that $\delta_\calT(t)$ is non-resonant with no roots on the
unit circle. Without loss of generality, we can choose $\lambda_j$ for $j=1,\dots,r$
with $|\lambda_j|<1$ for all $j$ such that the roots of $\delta_\calT(t)$ are
$\{\lambda_1^{\pm 1},\dots,\lambda_r^{\pm 1}\}$. 
Note that $p_{\calT,\ell}$ lies in $E[x][y]$ which is a subspace of
$E(x)_\loc[y]$ (where $E(x)_\loc$ denotes the ring of rational functions on $x$
which are regular at $x=0$). This fact, together with part (a) of the lemma, 
concludes the proof of the proposition.
\qed


\section{Examples}
\label{sec.examples}

The twisted loop invariants, defined as formal Gaussian integrals, are explicitly
computed algebraically in terms of the NZ-datum of an ideal triangulation. Likewise,
the twisted loop invariants can be computed algebraically using the twisted
version of the NZ-datum of~\cite{GY21}. In other words, Theorem~\ref{thm.feynman}
leads to an effective computation of the twisted loop invariants. 

In this section we illustrate our theorems by using an exact computation of the
$\ell$-loop invariants of $n$-cyclic covers for $\ell=2,3$ and various values
of $n$ depending on the complexity of the knot, i.e., on the number of tetrahedra,
the degree of its invariant trace field and the degree of its twisted $1$-loop
invariant $\delta(t)$. 

Recall that we can reconstruct the polynomial $p_{\calT,\ell}(x,y)$ from
$\calL$ and $(\ell-1)\binom{r+2\ell-2}{r}$ values of the $\ell$-loop invariants of
the $n$-cyclic covers. Note the relation $r \leq 3g-2$
(see~\cite[Thm.1.5]{Dunfield:twisted}) between the degree of the adjoint torsion
polynomial and the genus $g$ of the knot, namely the minimum genus of all Seifert
surfaces of the knot. In all of our examples, equality $r=3g-2$ is attained. 

The next table summarizes the number of values of $n$-cyclic covers needed to
determine the loop invariants for all $n$, when $\Lambda$ is known. 

\begin{table}[htpb!]
\label{table.gr}
\centering
  \begin{tabular}{c|c|c|c}
$g$ & $r$ & $\#2$-loop values & $\#3$-loop values \\ \hline
1 & 1 & 3 & 10 \\ \hline
2 & 4 & 15 & 140 \\ \hline 
$g$ & $3g-2$ & $\tfrac{3}{2}g(3g-1)$ & $\tfrac{1}{4}(3g+2)(3g+1)g(3g-1)$   
  \end{tabular}
\vspace{0.5cm}
\caption{The number of values needed to determine the 2 and 3-loop invariants
of cyclic covers when $\Lambda$ is known.}
\end{table}

\subsection{Genus 1 examples: $4_1$ and $5_2$ knots}
\label{sub.genus1}

In this section we give two examples of genus 1 knots, namely the two simplest
hyperbolic knots, the $4_1$ and the $5_2$ knot. For the knot $4_1$, 
we computed the 2 and 3-loop invariants of $n$-cyclic covers for $n=1,\ldots,100$.
Note that
these invariants are elements of the invariant trace field $\BQ(\sqrt{-3})$, but due to
the chirality of the knot, they are essentially elements of $\BQ$. We computed
those rational numbers numerically, and after multiplying them by the expected
denominators (coming from the $n$-cyclic resultant of a small power of
the twisted 1-loop invariant), we lifted the nearly rational numbers to exact rational
numbers, and checked that they agree within the precision of the computation (about
1000 digits). For the knot $5_2$, the 2 and 3-loop invariants of the $n$-cyclic
cover is an element of the cubic invariant trace field of discriminant $-23$.
In this case,
we numerically computed the invariants using Theorem~\ref{thm.feynman} for each
embedding of the shapes in the complex numbers, and then took the product thus
numerically computing numerically the coefficients of the miminal polynomial of
the invariants. The latter has rational coefficients, which as before can be
numerically computed and then lifted to exact rational numbers. Having done do,
we converted the geometric root of the minimal polynomial back to the invariant
trace field,
thus getting exact value of the 2 and 3-loop invariants for $n$-cyclic covers
of $5_2$ for $n=1,\dots,60$. Using this data, we then interpolated numerically
to find the formulas presented below. Once the formulas were found, an exact
computation can verify them for the computed values of $n$. 

The answer found agrees with the algebraic computation of Theorem~\ref{thm.quad}
and gives evidence to the conjecture that Theorem~\ref{thm.phi} works for all
natural numbers $n$, as opposed to all but finitely many $n$.

Although we computed the 2 and 3-loop invariants for cyclic covers for NZ data
that uses both the meridian and the longitude, we will present our results
using the longitude only. 



We now present our computations for the $4_1$ knot, whose invariant trace field is
$F=\BQ(\sqrt{-3})$ is a subfield of $\BC$. The twisted 1-loop invariant is
$\delta(t)=t-5+t^{-1}$, and has coefficients in the real part
of $F$, namely $\BQ$ (this is an accident because $4_1$ is amphichiral).
Let us denote the sum of the evaluation of a function at complex roots of unity
by $\Av_n(f(t))=\sum_{\omega^n=1} f(\omega)$. 
We have $\Phi_{\calT^{(n)},\ell}^\conn=\Av_n(\vphi_{\ell}^\conn(t,n))$ where
$\delta=\delta(t)$ and
\be
\label{41average}
\begin{aligned}
\vphi_{2}^\conn(t,n) &=
\Big(\frac{4}{3 n \delta^2} + \frac{20}{63 n \delta} + \frac{55}{1512} \Big)
\sqrt{-3} \\
\vphi_{3}^\conn(t,n) &=
-\frac{80}{3 n^2 \delta^4} - \frac{1976}{315 n^2 \delta^3}
+ \Big(-\frac{8}{189} + \frac{916}{1323 n^2}\Big)\frac{1}{\delta^2}
+ \Big(\frac{473}{26460} + \frac{2036}{19845 n^2}\Big)\frac{1}{\delta} \, .
\end{aligned}
\ee
This illustrates Theorem~\ref{thm.quad}.
The primes $67$ and $103$ that appear in the denominators of the
above expressions come from the $n$-cyclic resultant of $\delta(t)$ when one
determines the coefficients of $n^j \delta^j$ from few initial values of $n$. 


After doing a partial fraction decomposition of the rational functions that
appear in~\eqref{41average} and using Equations~\eqref{av1}-\eqref{av3}
of Example~\ref{ex.lemma}, we obtain explicit formulas for the invariants in terms
of generalized power sums illustrating Theorem~\ref{thm.phi}
\begin{small}
\be
\label{41short}
\begin{aligned}
\Phi_{\calT^{(n)},2}^\conn & =
-\frac{n (55 \lambda^n + 82 + 55 \lambda^{-n}) \, \sqrt{-3}}{
1512(1-\lambda^n)(1-\lambda^{-n})}
\\
\Phi_{\calT^{(n)},3}^\conn & =
\frac{-n^2 (\tfrac{32}{1323} \lambda^n + \tfrac{32}{441} + \tfrac{32}{1323}
  \lambda^{-n}) + n \sqrt{21} \big(-\tfrac{317}{238140}\lambda^{2n}
  -\tfrac{1985}{166698} \lambda^n +\tfrac{1985}{166698} \lambda^{-n}
+\tfrac{317}{238140}\lambda^{-2n}\big)}{
(1-\lambda^n)^2(1-\lambda^{-n})^2}
\end{aligned}
\ee
\end{small}
where $\lambda=\tfrac{1}{2}(5+\sqrt{21}) \approx 4.7912$ is one of the two roots
of $\delta(t)$, the other one being
$\lambda^{-1}=\tfrac{1}{2}(5-\sqrt{21}) \approx 0.2087$. 

It follows from either~\eqref{41average} or~\eqref{41short} that the leading
asymptotics of $\Phi_{\calT^{(n)},\ell}^\conn$ for $\ell=2,3$ are given by
\be
\label{41psi}
\begin{aligned}
\Phi_{\calT^{(n)},2}^\conn & = \frac{55 \sqrt{-3}}{1512} n + O(n |\lambda|^{-n}) \\
\Phi_{\calT^{(n)},3}^\conn & = -\frac{317 \sqrt{21}}{238140} n + O(n^2 |\lambda|^{-n}) 
\end{aligned}
\ee
illustrating Theorem~\ref{thm.psi}. Equations~\eqref{41short} imply that
the generating series of Proposition~\ref{prop.rational}  are given by
\be
\sum_{n=0}^\infty (1-\lambda^n)(1-\lambda^{-n}) \Phi^\conn_{\calT^{(n)},2} t^n =
-\sqrt{-3} \,
\frac{t(119 - 530 t + 1068 t^2 - 530 t^3 + 119 t^4)}{504 (-1 + t)^2 (1 - 5 t + t^2)^2}
\ee
and
\begin{tiny}
\begin{align*}
\sum_{n=0}^\infty ((1-\lambda^n)(1-\lambda^{-n}))^2  \Phi^\conn_{\calT^{(n)},3} t^n &=
\frac{t (1 + t)}{588 (-1 + t)^3 (1 - 23 t + t^2)^2 (1 - 5 t + t^2)^3} 
(343 - 15565 t + 249432 t^2 - 1448727 t^3
\\ & \hspace{-1cm}
+ 4346901 t^4 - 6772800 t^5 + 4346901 t^6 - 1448727 t^7 + 249432 t^8
- 15565 t^9 + 343 t^{10}) \,.
\end{align*}
\end{tiny}
The factor $1-23 t + t^2$ which appears in the above denominator
equals to $(1-\lambda^2t)(1-\lambda^{-2}t)$. This concludes the discussion
of the $4_1$ knot. 



The next hyperbolic knot is the $5_2$ knot, whose invariant trace field is the cubic
field of discriminant $-23$ generated by the root $\xi \approx -0.662 - 0.562\, i $
of the equation $\xi^3-\xi-1=0$. The twisted 1-loop invariant is
\be
\delta(t) = (2 + 4 \xi + 2 \xi^2) t -5 - 2 \xi + 3 \xi^2 + (2 + 4 \xi + 2 \xi^2) t^{-1}
\,.
\ee
We have $\Phi_{\calT^{(n)},\ell}^\conn =\Av_n(\vphi_{\ell}^\conn(t,n))$
where $\delta=\delta(t)$ and
\begin{tiny}
\begin{align*}
  \vphi_{2}^\conn(t,n) &=
  -\frac{4 (-16228 - 3232 \xi + 8679 \xi^2)}{7705 n} \frac{1}{\delta^2} \\
  & \quad
  + \Big( \frac{39 - 56 \xi - 24 \xi^2}{46}
  -\frac{4(-26127539 + 15044839 \xi + 3721992 \xi^2)}{94440185 n} \Big) \frac{1}{\delta}
  \\ & \quad
  +\frac{123094133 - 446744448 \xi + 259344006 \xi^2}{2266564440}
  \\
  \vphi_{3}^\conn(t,n) &=
  \Big(\frac{144171776}{516235} + \frac{86345584 \xi}{516235}
  - \frac{136288528 \xi^2}{516235} \Big) \frac{1}{n^2\delta^4}
  \\ & \quad
  \Big( \big(
  \frac{2021650619247678416}{12919632159420875}
  - \frac{194429261261137656 \xi}{12919632159420875}
  - \frac{1412467704798780848 \xi^2}{12919632159420875} \big) \frac{1}{n}
  \\ & \quad +
 \frac{362208}{7705} - \frac{57728 \xi}{7705} - \frac{243704 \xi^2}{7705}
 \Big)\frac{1}{n \delta^3}
 \\ & \quad
 \Big(
 \frac{19464170555699}{8731939505100} - \frac{1654507907596 \xi}{2182984876275}
 + \frac{661858625444 \xi^2}{727661625425}
 \\ & \quad + 
 \big(\frac{43685924340213}{2910646501700} - \frac{13472722690929 \xi}{1455323250850}
     + \frac{2008555368111 \xi^2}{2910646501700} \big)\frac{1}{n}
 \\ & \quad +
 \big(
 \frac{161627626755245606632}{8963543285548396125} -
 \frac{119998551075098128112 \xi}{8963543285548396125} +
 \frac{127286031414479468 \xi^2}{2987847761849465375}
 \big) \frac{1}{n^2}
 \Big) \frac{1}{\delta^2}
 \\ & \quad
 \Big(
 \big(
 \frac{-19378724062777204444}{475067794134064994625}
 - \frac{474092286600322084396 \xi}{475067794134064994625}
 + \frac{128349517059147927744 \xi^2}{158355931378021664875}
 \big) \frac{1}{n^2}
  \\ & \quad
 +
 \big(
 \frac{1746056639554239}{35675794171336900}
 - \frac{31637787802490587 \xi}{17837897085668450}
 + \frac{11693862723463677 \xi^2}{8918948542834225}
 \big) \frac{1}{n}
  \\ & \quad
 +
 \big(
 \frac{59134987864619182444}{475067794134064994625}
 - \frac{1527155471544628788041 \xi}{1900271176536259978500}
 + \frac{3826051934183205772 \xi^2}{6885040494696594125}
 \big) 
 \Big) \frac{1}{\delta}
 \\ & \quad
 \Big(
 -\frac{428855832942393}{8918948542834225}
 -\frac{2998162280908073 \xi}{35675794171336900}
 +\frac{1615737458359533 \xi^2}{17837897085668450}
 \Big) \,.
\end{align*}
\end{tiny}

The leading asymptotic values are given by $\Phi_{\calT^{(n)},\ell}^\conn
= n \Psi_{\calT,\ell}^\conn + O(n^{\ell-1} |\lambda|^{n})$ for $\ell=2,3$ where
$\lambda \approx 0.0502 -0.1704 \,i$ is a root of $\delta(t)$ and satisfies the
equation $8 \lambda^6 -28 \lambda^5 + 270 \lambda^4 -109 \lambda^3
+ 270 \lambda^2 -28 \lambda +8=0$ and 


\begin{tiny}
\be
\label{52psi}
\begin{aligned}
\Psi_{\calT^{(n)},2}^\conn & = \frac{1}{1160480993280}(-16601383280
+ 239466164328 \lambda - 30998500743 \lambda^2 + 51073175277 \lambda^3
\\ & \qquad
- 2600093877 \lambda^4 + 384393303 \lambda^5)
\\  
\Psi_{\calT^{(n)},3}^\conn & = \frac{1}{26122110666289424422956544000}(
3763333983996990578027312 - 27832672813601695938777064 \lambda 
\\ & \qquad
+ 98732772027957178344155 \lambda^2 + 1194221340324541487037559 \lambda^3 - 
453984084634619809746255 \lambda^4
\\ & \qquad
+ 29998843726647510986933 \lambda^5) \,.
\end{aligned}
\ee
\end{tiny}


\subsection{Higher genus examples: the $6_2$ and the $(-2,3,7)$ pretzel
  knots}
\label{sub.62.237}


The above two examples illustrate our main theorems when $\delta(t)$ is quadratic
and the Seifert genus $g$ is $1$. We now present two further examples of higher
genus illustrating Theorem~\ref{thm.phi} and Equation~\eqref{2loopn}.
The coefficients $c_{ij}$, $c_i$ and $c_0$ in Equation~\eqref{2loopn} 
can be determined from $(r+1)(r+2)/2$
values of $\Phi_{\mathcal{T}^{(n)},2}$ and lie in the splitting field $E$ of $\delta(t)$. 

The first example is the $6_2$ knot where $g=2$ and $r=4$. Its default \texttt{SnapPy}
triangulation has $5$ tetrahedra and its invariant trace field has degree $5$.
We computed $120$ exact values of the 2-loop invariant of the cyclic covers.
To do so, we numerically computed these values for all $5$ embeddings of the shapes to
the complex numbers, and from that we computed the minimal polynomial
(whose coefficients are integers computed approximately to high precision and then
recognized). Once we knew the minimal polynomial, we converted its chosen root to
the fixed embedding of the invariant trace field to the complex numbers. Having done
so, we used $15$ values of $n=1,\dots,15$ to numerically compute the above coefficients
to $1000$ digits, and then used the remaining $105$ values of $n$ to check our
numerical answer, which agreeded to all $1000$ digits of precision.
Note that the coefficients
$c_{ij}$, $c_i$ and $c_0$ that appear in~\eqref{2loopn} are elements of the splitting
field of $\delta(t)$, an explicit number field of degree $120$. Although we can
numerically compute the coefficients to arbitrary high precision, eg. $10000$ digits,
it is not likely that we can express them explicitly by elements of the splitting
field. 

Our second example is the $(-2,3,7)$ pretzel knot where $g=5$ and $r=13$. 
Its default \texttt{SnapPy} triangulation has $3$ ideal tetrahedra and
its invariant trace field is cubic (and equal to that of the $5_2$ knot).
Working as above, we were able to compute the first $140$ values of the $2$-loop
invariant of its $n$-cyclic cover. We used $105$ values to determine the coefficients
of~\eqref{2loopn} to the precision of $1000$ digits, and then $35$ further values
to check our prediction. Once again, Equation~\eqref{2loopn} worked to all
the accuracy of $1000$ digits. In this case, the complexity of the splitting field
is prohibitive, and it is unlikely that one will be able to compute the exact
values of the $105$ constants in Equation~\eqref{2loopn}.

\subsection*{Acknowledgments}
S.G. wishes to thank Efim Zelmanov for enlightening conversations.
S.Y. wishes to thank Jaeseong Oh for helpful comments on Section~\ref{sub.flows}.


\bibliographystyle{hamsalpha}
\bibliography{biblio}
\end{document}